\documentclass[a4paper,11pt,reqno, english]{amsart}  
\usepackage[utf8]{inputenc}
\usepackage[T1]{fontenc}
\usepackage{amsmath,amsthm}
\usepackage{amsfonts,amssymb,enumerate}
\usepackage{url,paralist}
\usepackage{mathtools}  
\usepackage[colorlinks=true,urlcolor=blue,linkcolor=red,citecolor=magenta]{hyperref}
\usepackage{enumerate}
\usepackage{tikz}
\usepackage[left=1in,right=1in,top=1in,bottom=1in]{geometry}

\theoremstyle{plain}
\newtheorem{thm}{Theorem}[section]
\newtheorem{lem}[thm]{Lemma}
\newtheorem{cor}[thm]{Corollary}

\newtheorem*{thm*}{Theorem}

\theoremstyle{definition}

\newtheorem{ex}[thm]{Example}
\newtheorem{rem}[thm]{Remark}

\newcommand{\Z}{\mathbb{Z}}
\newcommand{\R}{\mathbb{R}}
\renewcommand{\emptyset}{\varnothing}

\DeclareMathOperator{\conv}{\mathrm{conv}}

\begin{document}

\title{Covering and labeling generalizations of the Borsuk--Ulam theorem}


\author[Frick]{Florian Frick}
\address[FF]{Dept.\ Math.\ Sciences, Carnegie Mellon University, Pittsburgh, PA 15213, USA}
\email{frick@cmu.edu} 

\author[Wellner]{Zoe Wellner}
\address[ZW]{School of Math. and Stat.\ Sciences, Arizona State University, Tempe, AZ 85287, USA}
\email{zwellner@asu.edu}

\thanks{FF and ZW were supported by NSF CAREER Grant DMS 2042428.}

\date{\today}

\begin{abstract}
\small 
We prove multiple generalizations of Fan's combinatorial labeling result for sphere triangulations. This can be seen as a comprehensive extension of the Borsuk--Ulam theorem. 
In typical applications, the Borsuk--Ulam theorem gives complexity bounds in a suitable sense, whereas our extension additionally provides insight into the structure of objects satisfying the complexity bound.
This structure is governed by order types of finite point sets in Euclidean space and more generally by the intersection combinatorics of faces under continuous maps from the simplex.
We develop some of those applications for sphere coverings, Kneser-type colorings, Hall-type results for hypergraphs, and hyperplane mass partitions, among other consequences. We provide a new proof of the topological Hall theorem and extend it into a result that simultaneously generalizes hypergraph Hall theorems and topological lower bounds for chromatic numbers.
\end{abstract}

\maketitle


\section{Introduction}

The Borsuk--Ulam theorem~\cite{Borsuk1933} has found numerous applications across mathematical disciplines from geometric partitioning results in convex geometry~\cite{roldansoberon2022survey}, chromatic numbers in combinatorics~\cite{matousek2004topological}, consensus-halving and fair division~\cite{simmons2003consensus}, fixed-point theorems in topology and non-linear analysis~\cite{su1997borsuk, volovikov2008borsuk}, as well as the minimax variational inequalities frequently used in game theory derived from these fixed point theorems~\cite{ bartsch2006topological,fan1972minimax}, existence results for solutions of nonlinear PDEs~\cite{rabinowitz1982multiple}, dissimilarity~\cite{GH-BU-VR, lim2021gromov} and distortion~\cite{badoiu2005approximation} in metric geometry, and inscribability~\cite{aslam2020splitting} and incidence~\cite{Bajmoczy1979} problems in geometric topology, to problems concerning algorithmic complexity of approximate graph coloring~\cite{austrin2020improved} and PAC learning~\cite{chase2024local, hatami2023borsuk} in theoretical computer science. See Matou\v sek's book for some applications~\cite{Matousek2003book}. The Borsuk--Ulam theorem states that any continuous map $f\colon S^d \to \R^d$ defined on the $d$-sphere~$S^d$ has a zero, provided that $f$ is \emph{odd}, that is, $f(-x) = -f(x)$ for all $x\in S^d$. In typical applications of this result, the zero of~$f$ corresponds to a desired solution of the problem. Fan proved a far-reaching generalization of the Borsuk--Ulam theorem that can be phrased as follows:

\begin{thm}[Fan~\cite{Fan1952}]
\label{thm:fan}
    Let $m$ be a positive integer, and let $A_1, \dots, A_m \subset S^d$ be closed sets such that $A_i \cap (-A_i) =\emptyset$ for all $i \in [m]$ and $\bigcup_i A_i \cup \bigcup_i (-A_i) = S^d$. Then there are indices $i_1 < i_2 < \dots < i_{d+1}$ such that $\bigcap_{j = 1}^{d+1} (-1)^j A_{i_j} \ne \emptyset$.
\end{thm}

For $f\colon S^d \to \R^m$ an odd map without zeros, let $A_i = \{x \in S^d\, |\, f_i(x) \ge \varepsilon\}$. By compactness of~$S^d$, the $A_i$ and their negatives~$-A_i$ cover the sphere for sufficiently small~$\varepsilon$. Fan's result thus provides structural insight into odd maps without zeros: They do not exist for $m\le d$, and for $m>d$ Fan derives a result about the simultaneous maximization of the absolute value of some $d+1$ coordinates of~$f$.

We explain the utility of this extension of the Borsuk--Ulam theorem in an example: The Borsuk--Ulam theorem yields a lower bound for the chromatic number~$\chi(G)$ of a graph~$G$; Fan's theorem gives structural results for proper colorings of $G$ with (necessarily) at least~$\chi(G)$ colors; see~\cite{fan1982evenly, hajiabolhassan2011generalization, SimonyiTardos2006, simonyi2007colorful} for various such examples. It is thus of interest to study the general space of results that extend Fan's theorem, since those results govern the structure of a variety of mathematical problems---those that can be solved through an application of the Borsuk--Ulam theorem. As $m-d$ grows the combinatorial patterns of indices and signs that yield non-empty intersections as in Fan's theorem become surprisingly rich and are determined by Radon-type (convex hull) intersection results of $m$ points in~$\R^{d-1}$. Concretely we prove:

\begin{thm}
\label{thm:fan-gen}
    Let $m$ be a positive integer, and let $X = \{x_1, \dots, x_m\} \subset \R^{d-1}$ be a set of $m$ points. Let $A_1, \dots, A_m \subset S^d$ be closed sets such that $A_i \cap (-A_i) =\emptyset$ for all $i \in [m]$ and $\bigcup_i A_i \cup \bigcup_i (-A_i) = S^d$. Then there are disjoint subsets $S, T \subset [m]$ such that \[\conv\{x_i\, |\, i \in S\} \cap \conv\{x_i\, |\, i \in T\} \ne \emptyset \ \text{and} \ \bigcap_{i \in S} A_i \cap \bigcap_{i \in T} (-A_i) \ne \emptyset.\]
\end{thm}

Fan's result follows from placing $X$ in cyclic position~\cite{gale1963neighborly}. Thus Fan's theorem corresponds to a single order type of $m$ points as in Theorem~\ref{thm:fan-gen}, but every such order type, of which there are many~\cite{goodman1986upper}, yields another variant of Fan's theorem.
In fact, Fan proves an equivalent result for labelings of sphere triangulations, and shows moreover that the number of facets that are alternatingly labeled is odd. We also prove a strengthening of the corresponding labeling result; see Theorem~\ref{thm:fan-gen-label}. We then prove a colorful generalization of Theorem~\ref{thm:fan-gen}, see Theorem~\ref{thm:fan-gen-col} extending our earlier work~\cite{FrickWellner2023}. While we phrase Theorem~\ref{thm:fan-gen} as intersection combinatorics of convex hulls constraining intersections of closed sets in sphere coverings (or in the contrapositive, intersection combinatorics in sphere coverings constraining intersections of convex hulls), our result admits a ``continuous generalization,'' where convex hulls are replaced with continuous images under a map from the simplex on~$X$; see Theorem~\ref{thm:fan-gen-cont}.

Our proof of Theorem~\ref{thm:fan-gen} is a short and simple reduction to the Borsuk--Ulam theorem. Even for the special case of Fan's theorem our proof might be the simplest in the literature; compare~\cite{LoeraGoaocMeunierMustafa2017, musin2016} for a simple proof of Fan's theorem in the case $m=d+1$. In particular, our proof of finding an approximate point of intersection in the full generality of Theorems~\ref{thm:fan} and~\ref{thm:fan-gen} is constructive and exhibits the corresponding search problem as lying in PPA~\cite{aisenberg20202}; see~\cite{prescott2005constructive} for a constructive proof of Fan's theorem and~\cite{freund1981constructive} for a constructive proof of Tucker's lemma. 

In Section~\ref{sec:proof} we present the proofs of the main results. In Sections~\ref{sec:consequences-top}, \ref{sec:consequences-comb}, and~\ref{sec:consequences-geom} we collect consequences of these results, that is, for various applications of the Borsuk--Ulam theorem, we develop results about the structure of solutions that go beyond mere non-existence. Section~\ref{sec:consequences-top} collects topological consequences about sphere coverings, embeddings, odd maps without common zeros, and connections to related fixed point theorems; Section~\ref{sec:consequences-comb} gives combinatorial consequences, and Section~\ref{sec:consequences-geom} generalizations of the ham sandwich theorem in discrete geometry. Concretely,
\begin{compactitem}
    \item we deduce results about the intersection combinatorics of sphere coverings such as a colorful generalization of Fan's theorem, see Corollary~\ref{cor:fan-col}, a colorful Lusternik--Schnirelman theorem, see Corollary~\ref{cor:localLS-signs}, and a colorful generalization of the local Lusternik--Schnirelman theorem of Chase, Chornomaz, Moran, and Yehudayoff~\cite{chase2024local}, see Corollary~\ref{cor:localLS}.
    \item We exhibit various non-embeddability results for simplicial complexes as special cases of Theorem~\ref{thm:fan-gen-cont}, the continuous extension of Theorem~\ref{thm:fan-gen}, see Subsection~\ref{subsec:radon}.
    \item We prove a generalization of Fan's theorem about the existence of points, where a set of odd real-valued maps is minimized or maximized; see Corollary~\ref{cor:odd}.
    \item Subsection~\ref{subsec:fixedpt} briefly surveys fixed-point theorems, such as Komiya's theorem~\cite{Komiya1994}, that our main result generalizes to the equivariant setting. 
    \item In Subsection~\ref{subsec:kneser}, we give general results on the existence of rainbow faces (see Corollary~\ref{cor:rainbow}) that generalize topological lower bounds for chromatic numbers of graphs, the topological Hall theorem, and results about the existence of colorful complete bipartite subgraphs. In particular, we give a new equivariant-topological proof of the topological Hall theorem.
    \item We derive a Hall-type theorem for hypergraphs in terms of systems of disjoint representatives; see Corollary~\ref{cor:hall-zigzag}.
    \item In Section~\ref{sec:consequences-geom}, we present generalizations of the Ham Sandwich theorem to more measures than dimensions, see Corollary~\ref{cor:fan-ham}, and a colorful extension; see Corollary~\ref{cor:fan-ham-col}.
\end{compactitem}

In Section~\ref{sec:two-spheres}, we extend our results to a product of two spheres and $(\Z/2)^2$-equivariance. We give an application to mass partitions by two hyperplanes. In Section~\ref{sec:final}, we discuss the generality and optimality of our results.

There is a large variety of variants and generalizations of the Borsuk--Ulam theorem; see Steinlein's survey~\cite{steinlein1985borsuk} of some of these results and Blagojevi\'c and Ziegler~\cite{blagojevic2017} for a more recent survey.

\section{Proof of the main result and three extensions}
\label{sec:proof}

Here we give a short proof of Theorem~\ref{thm:fan-gen}. We then prove three extensions:
\begin{compactenum}[1.]
    \item An extension to odd labelings of the vertex set of sphere triangulations, where combinatorics of facet labels replace the intersections combinatorics of Theorem~\ref{thm:fan-gen}. This extension in addition to the existence asserts that the number of facets that are appropriately labeled is odd, extending Fan's original labeling result; see Theorem~\ref{thm:fan-gen-label}.
    \item Whereas Theorem~\ref{thm:fan-gen} gives conditions for an intersection of two convex hulls, the continuous extension, Theorem~\ref{thm:fan-gen-cont}, more generally asserts that for any continuous map from the simplex on the set $X$ to~$\R^{d-1}$, the images of the corresponding faces of the simplex will intersect.
    \item Theorem~\ref{thm:fan-gen-col} gives a colorful generalization of Theorem~\ref{thm:fan-gen} that applies to $d+1$ sphere coverings instead of one. If all sphere coverings are the same this recovers Theorem~\ref{thm:fan-gen}.
\end{compactenum}

\medskip

Let $X \subset \R^{d-1}$. A pair~$(A,B)$ of disjoint subsets of~$X$ is called \emph{Radon} (for~$X$) if their convex hulls intersect, $\conv A \cap \conv B \ne \emptyset$. We will need the following standard correspondence between Radon pairs of a point set in~$\R^{d-1}$ and certain point sets in~$\R^d$ whose convex hulls capture~$0$; see Radon~\cite{radon1921mengen}. For a set $X \subset \R^{d-1}$ write $X_+ = \{(x,1) \in \R^d\, |\, x\in X\}$ and $X_- = \{(-x,-1) \in \R^d\, |\, x\in X\}$.

\begin{lem}
\label{lem:radon}
    Let $X \subset \R^{d-1}$, and let $A,B\subset X$ be disjoint subsets. The pair~$(A,B)$ is Radon if and only if $0 \in \conv(A_+\cup B_-)$.
\end{lem}

\begin{proof}
    The convex hulls of $A$ and $B$ intersect if and only if there are coefficients $\lambda_a \ge 0$, $a \in A$, and $\lambda_b \ge 0$, $b \in B$, with $\sum_{a\in A} \lambda_a = 1 = \sum_{b \in B} \lambda_b$ and $\sum_{a\in A} \lambda_a a = \sum_{b\in B} \lambda_b b$. Then 
    \[
        0 = \sum_{a \in A} \frac12\lambda_a (a,1) + \sum_{b \in B} \frac12\lambda_b (-b,-1),
    \]
    and so $0 \in \conv(A_+\cup B_-)$. 

    Conversely, if $0 = \sum_{a \in A} \lambda_a (a,1) + \sum_{b\in B} \lambda_b (-b,-1)$ with $\sum_{a \in A} \lambda_a + \sum_{b \in B} \lambda_b = 1$. The last coordinate implies $\sum_a \lambda_a = \frac12 = \sum_b \lambda_b$. Thus $\sum_a 2\lambda_a a = \sum_b 2 \lambda_b b$ with $\sum_a 2\lambda_a = 1 = \sum_b 2 \lambda_b$, and so $\conv A \cap \conv B \ne \emptyset$. 
\end{proof}

By a \emph{triangulation} of a space~$X$, we always mean a simplicial complex whose geometric realization is homeomorphic to~$X$. See~\cite{kozlov2008, Matousek2003book} for the basics.

\begin{proof}[Proof of Theorem~\ref{thm:fan-gen}]
    Let $\varepsilon > 0$ such that any $A_i$ is at distance at least~$\varepsilon$ from~$-A_i$. Let $\Sigma$ be an antipodally symmetric triangulation of~$S^d$, where every facet has diameter less than~$\varepsilon$. Let $V$ denote the vertex set of~$\Sigma$ and let $f \colon V \to \R^d$ be a function with the property that $f(v) = (x_i, 1)$ for $v \in A_i$ and $f(v) = (-x_i, -1)$ for $v \in (-A_i)$. If $v$ is in multiple~$A_i$, then choose one arbitrarily, but in such a way that $f(-v) = -f(v)$. 

    We can think of $f$ as a map $f\colon \Sigma \to \R^d$ by linearly extending it to the faces of~$\Sigma$. The zeros of~$f$ then precisely correspond to the Radon pairs for~$X$ by Lemma~\ref{lem:radon}. Here we use that by the choice of~$\varepsilon$ if $f(v) = -f(w)$ then $v$ and~$w$ are not in a common face of~$\Sigma$. By the Borsuk--Ulam theorem, $f$ has a zero. Now let $\varepsilon$ go to zero and use compactness of~$S^d$. 
\end{proof}

Let $\Sigma$ be an antipodally symmetric triangulation of~$S^d$ on vertex set~$V$. 
Let $f \colon V \to \R^d$ be a labeling of the vertices of~$\Sigma$ with points in~$\R^d$ that respects the antipodal symmetry: $f(-v) = -f(v)$ for all~$v \in V$. We will refer to $f$ as an \emph{odd} labeling. A facet $\sigma$ of $\Sigma$ \emph{captures~$0$} (for~$f$) if $0 \in \conv\{f(v)\, |\, v\in \sigma\}$. A labeling $f\colon V \to \R^d$ is \emph{generic} if for every face~$\tau$ of dimension at most~${d-1}$, we have that $0 \notin \conv\{f(v)\, |\, v\in \tau\}$. 

\begin{thm}
\label{thm:fan-gen-label}
    Let $f\colon V \to \R^d$ be a generic odd labeling of the vertex set~$V$ of an antipodally symmetric triangulation~$\Sigma$ of~$S^d$. Then the number of facets of $\Sigma$ that capture $0$ is~$2k$, where $k$ is odd. 
\end{thm}

\begin{proof}
    Let $\widetilde f \colon \Sigma \to \R^{d+1}$ be the map that is the linear extension to faces of~$f$ in the first $d$ coordinates, and in the last coordinate $\widetilde f_{d+1}$ is defined as follows: $\widetilde f_{d+1}$ is identically $0$ on every face that does not capture~$0$ for~$f$; facets that capture~$0$ for~$f$ come in antipodal pairs $\sigma, -\sigma$. The $(d+1)$st coordinate $\widetilde f_{d+1}$ is positive in the interior of~$\sigma$, and thus negative in the interior of~$-\sigma$. (The choice of which facet is $\sigma$ and which is $-\sigma$ is arbitrary.) Concretely, we map the barycenter of~$\sigma$ to the standard basis vector~$e_{d+1}$ and extend along rays to all of~$\sigma$. In particular, for $x \in \widetilde f^{-1}(e_{d+1})$ there is a neighborhood $U$ of~$x$ such that $\widetilde f|_U$ is a homeomorphism. Since $\widetilde f$ is a local homeomorphism around points in the preimage of~$e_{d+1}$, its degree $\deg \widetilde f$ can be computed as a sum of local degrees, which are all~$\pm 1$, around these preimages; see~{\cite[Prop.~2.30]{hatcher2002}}. 
    The degree $\deg\widetilde f$ is odd by the Borsuk--Ulam theorem, so $\widetilde f^{-1}(e_{d+1})$ has odd cardinality, and twice as many facets of~$\Sigma$ capture~$0$.
\end{proof}

Theorem~\ref{thm:fan-gen-label} is the labeling variant of Theorem~\ref{thm:fan-gen} as we explain in the following remark:

\begin{rem}
    Let $X = \{x_1, \dots, x_m\} \subset \R^{d-1}$ be generic in the sense that any Radon pair $(A,B)$ for~$X$ involves at least~$d+1$ points, and let $\Sigma$ be an antipodally symmetric triangulation of~$S^d$ on vertex set~$V$. Let $A_1, \dots, A_m \subset S^d$ be closed sets such that $A_i \cap (-A_i) =\emptyset$ for all $i \in [m]$ and $\bigcup_i A_i \cup \bigcup_i (-A_i) = S^d$. Define $f\colon V \to \R^d$ by first finding the smallest~$i$ such that $v\in A_i \cup (-A_i)$ and letting $f(v) = (x_i,+1)$ for $v\in A_i$ and $f(v)=(-x_i,-1)$ for $v\in (-A_i)$. This map satisfies $f(-v) = -f(v)$. By Lemma~\ref{lem:radon} $f$ is a generic labeling. 

    By Theorem~\ref{thm:fan-gen-label} there are $2k$ facets of~$\Sigma$, for some odd integer~$k$, that capture~$0$. Let $v_1, \dots, v_{d+1}$ be the vertices of one such facet~$\sigma$. Let $S\subset [d+1]$ be the set of $i \in [d+1]$ where the last coordinate of $f(v_i)$ is~$+1$, and let $T = [d+1]\setminus S$. Say $f(v_j) = (x_{i_j},+1)$ for $j \in S$ and $f(v_j) = (-x_{i_j},-1)$ for $j \in T$. Then $\sigma$ intersects the sets $A_{i_j}$ for $j \in S$ and $-A_{i_j}$ for $j \in T$. Further, by Lemma~\ref{lem:radon} $\conv\{x_{i_j} \, | \, j \in S\} \cap \conv\{x_{i_j} \, | \, j \in T\} \ne \emptyset$.
\end{rem}

For Theorem~\ref{thm:fan-gen-label}, we need a genericity assumption to derive a constraint on the parity of facets that capture~$0$. Without this genericity assumption we still get an existence result, which we may phrase in terms of Radon partitions as in Theorem~\ref{thm:fan-gen}. Thus the following lemma is a labeling variant of Theorem~\ref{thm:fan-gen}, which we record here for later use:

\begin{lem}
\label{lem:labeled-facet}
    Let $\Lambda$ be a simplicial complex on vertex set~$V \times \{-1,+1\}$ such that $(v,+1) \mapsto (v,-1)$ induces a well-defined $\Z/2$-action on~$\Lambda$. Assume there is a $\Z/2$-equivariant map $S^d \to \Lambda$. Let $f\colon V \times \{-1,+1\} \to \R^{d-1}$ be a map with $f(v,-1) = f(v,+1)$ for all~$v\in V$. Then there is a face~$\sigma$ of~$\Lambda$ such that \[\conv f(\sigma \cap (V\times \{+1\})) \cap \conv f(\sigma \cap (V \times \{-1\})) \ne \emptyset.\]
\end{lem}

\begin{proof}
    Define $h\colon V \times \{-1,+1\} \to \R^d$ by $h(v,+1) = (f(v),+1)$ and $h(v,-1) = (-f(v), -1)$. By linearly extending onto faces $h$ is a $\Z/2$-equivariant map $\Lambda \to \R^d$. By composition we obtain a $\Z/2$-equivariant map $S^d \to \Lambda \to \R^d$, and thus $h$ has a zero by the Borsuk--Ulam theorem. Let $\sigma$ be a face of $\Lambda$ that contains a zero of~$h$. Now Lemma~\ref{lem:radon} finishes the proof.
\end{proof}

Denote the standard $(m-1)$-dimensional simplex by
\[
    \Delta_{m-1} = \{x\in \R^m\, |\, x_i \ge 0 \ \text{for all}\ i \ \text{and} \ \sum_i x_i =1\}.
\]
The vertices of $\Delta_{m-1}$ are the standard basis vectors $e_1, \dots, e_m$, and we will identify this vertex set with~$[m]$. We will denote the boundary of the $m$-dimensional crosspolytope by
\[
    (\Delta_{m-1})^{*2}_\Delta = \{x \in \R^m\, |\, \sum_i |x_i|=1\}.
\]
The notation indicates that the crosspolytope~$(\Delta_{m-1})^{*2}_\Delta$ is the deleted join of the simplex~$\Delta_{m-1}$. Note that $\Delta_{m-1} \subset (\Delta_{m-1})^{*2}_\Delta$ and $-\Delta_{m-1} \subset (\Delta_{m-1})^{*2}_\Delta$. Every point in~$(\Delta_{m-1})^{*2}_\Delta \setminus (\Delta_{m-1} \cup (-\Delta_{m-1}))$ is a unique convex combination of a point in~$\Delta_{m-1}$ and a point in~$-\Delta_{m-1}$, and these points lie in faces $\sigma$ of~$\Delta_{m-1}$ and $-\tau$ of~$-\Delta_{m-1}$ such that $\sigma$ and $\tau$ are disjoint.

\begin{thm}
\label{thm:fan-gen-cont}
    Let $m$ be a positive integer, and let $h\colon \Delta_{m-1} \to \R^{d-1}$ be a continuous map. Let $A_1, \dots, A_m \subset S^d$ be closed sets such that $A_i \cap (-A_i) =\emptyset$ for all $i \in [m]$ and $\bigcup_i A_i \cup \bigcup_i (-A_i) = S^d$. Then there are disjoint faces $\sigma$ and $\tau$ of~$\Delta_{m-1}$ such that \[h(\sigma) \cap h(\tau) \ne \emptyset \ \text{and} \ \bigcap_{i \in \sigma} A_i \cap \bigcap_{i \in \tau} (-A_i) \ne \emptyset.\]
\end{thm}

\begin{proof}
    Let $\varepsilon > 0$ such that any $A_i$ is at distance at least~$\varepsilon$ from~$-A_i$. Let $\Sigma$ be an antipodally symmetric triangulation of~$S^d$, where every facet has diameter less than~$\varepsilon$. Let $q\colon S^d \to \R P^d$ be the quotient map that identifies pairs of antipodes. Let $\alpha_1, \dots, \alpha_m \colon \R P^d \to [0,1]$ be a partition of unity subordinate to the cover $q(A_1), \dots, q(A_m)$, that is, the $\alpha_i$ are continuous maps with $\sum_{i=1}^m \alpha_i(x) = 1$ for all $x \in \R P^d$ and $\alpha_i(x) > 0$ implies that $x \in q(A_i)$; see~\cite{Rudin1987book}. 
    For every $i \in [m]$ let $\widetilde\alpha_i\colon S^d \to [-1,1]$ be defined as $\widetilde \alpha_i(x) = \alpha_i(q(x))$ if $x \in A_i$ and $\widetilde\alpha_i(x) = -\alpha_i(q(x))$ otherwise. This is a well-defined, continuous map since $A_i \cap (-A_i) =\emptyset$ and $\alpha_i(q(x)) = 0$ for every $x\in S^d \setminus (A_i \cup(-A_i))$.

    Let $\widetilde h \colon (\Delta_{m-1})^{*2}_\Delta \to \R^d = \R^{d-1} \times \R$ be defined by 
    \[
        \widetilde h(\lambda x + (1-\lambda)y) = (\lambda h(x) - (1-\lambda)h(-y), \lambda - (1-\lambda)), 
    \]
    that is, $\widetilde h(x) = (h(x), 1)$ for $x\in \Delta_{m-1}$, $\widetilde h(x) = (-h(x), -1)$ for $x\in (-\Delta_{m-1})$, and $\widetilde h$ interpolates linearly in between using that every other point is a unique convex combination $\lambda x + (1-\lambda)y$ of points $x\in \Delta_{m-1}$ and $y\in \Delta_{m-1}$. In particular, $h$ is an odd map.
    The map 
    \[
        \widetilde \alpha\colon S^d \to \R^m, \ \widetilde\alpha(x) = (\widetilde\alpha_1(x), \dots, \widetilde\alpha_m(x))
    \]
    is odd and its image is contained in~$(\Delta_{m-1})^{*2}_\Delta$.

    The composition $\widetilde h \circ \widetilde \alpha \colon S^d \to \R^d$ is odd and thus has a zero~$x_0 \in S^d$ by the Borsuk--Ulam theorem. The point~$\widetilde \alpha(x_0)$ lies in the relative interior of a face of~$(\Delta_{m-1})^{*2}_\Delta$ that is the convex hull of a face $\sigma$ of~$\Delta_{m-1}$ and a face $-\tau$ of~$-\Delta_{m-1}$, where $\sigma$ and $\tau$ are disjoint; say $\widetilde\alpha(x_0) = \lambda x + (1-\lambda)y$ for $x\in \sigma$ and $y \in (-\tau)$. The last coordinate constrains $\lambda= \tfrac12$ and $\widetilde\alpha(x_0)$ lies halfway between~$\Delta_{m-1}$ and~$(-\Delta_{m-1})$. Since $0 = \tfrac12h(x)-\tfrac12h(-y)$, we have that $h(x) = h(-y)$, and thus $h(\sigma) \cap h(\tau) \ne \emptyset$. By definition of~$\widetilde \alpha$ we have that $\bigcap_{i \in \sigma} A_i \cap \bigcap_{i \in \tau} (-A_i) \ne \emptyset$.
\end{proof}

To establish the following colorful generalization of Theorem~\ref{thm:fan-gen}, we repeat the proof with a small modification on the barycentric subdivision of a fine triangulation~$\Sigma$ of~$S^d$, where now the dimension of the face that a vertex subdivides determines which covering is used to define the value of the map~$f$ at that vertex. This approach was used by Su~\cite{Su1999} to prove a colorful generalization of Sperner's lemma. 
Recall that the \emph{barycentric subdivision} $\Sigma'$ of a simplicial complex~$\Sigma$ has the faces of $\Sigma$ as its vertex set, and $\{\sigma_1, \dots, \sigma_k\}$ is a face of~$\Sigma'$ if (after possibly reordering) $\sigma_1 \subset \dots \subset \sigma_k$. The barycentric subdivision~$\Sigma'$ is homeomorphic to~$\Sigma$.

\begin{thm}
\label{thm:fan-gen-col}
    Let $m$ be a positive integer, and let $X^{(j)} = \{x_1^{(j)}, \dots, x_m^{(j)}\} \subset \R^{d-1}$, $j\in [d+1]$, be sets of $m$ points. Let $A_1^{(j)}, \dots, A_m^{(j)} \subset S^d$ be closed sets for $j \in [d+1]$ such that $A_i^{(j)} \cap (-A_i^{(\ell)}) =\emptyset$ for all $i \in [m]$ and $j \ne \ell\in[d+1]$ and such that $\bigcup_i A_i^{(j)} \cup \bigcup_i (-A_i^{(j)}) = S^d$ for all $j\in [d+1]$. Then there are disjoint subsets $S, T \subset [m]$ and an injective map $\pi\colon S\cup T \to [d+1]$ such that 
    \[\conv\{x_i^{(\pi(i))}\, |\, i \in S\} \cap \conv\{x_i^{(\pi(i))}\, |\, i \in T\} \ne \emptyset \ \text{and} \ \bigcap_{i \in S} A_i^{(\pi(i))} \cap \bigcap_{i \in T} (-A_i^{(\pi(i))}) \ne \emptyset.\]
\end{thm}

\begin{proof}
    Let $\varepsilon > 0$ such that any $A_i^{(j)}$ is at distance at least~$\varepsilon$ from~$\bigcup_{\ell\ne j} (-A_i^{(\ell)})$. Let $\Sigma$ be an antipodally symmetric triangulation of~$S^d$, where every facet has diameter less than~$\varepsilon$. Let $V$ denote the vertex set of the barycentric subdivision~$\Sigma'$ of~$\Sigma$, that is, there is a one-to-one correspondence between faces $\sigma$ of $\Sigma$ and vertices~$v_\sigma \in V$, where $v_{\sigma_1}, \dots, v_{\sigma_k}$ form a face of the barycentric subdivision of~$\Sigma$ whenever (after possibly reordering the~$\sigma_i$) $\sigma_1 \subset \sigma_2 \subset \dots \subset \sigma_k$. In particular, the vertices of any face of~$\Sigma'$ correspond to faces of~$\Sigma$ that have pairwise distinct dimensions. Let $f \colon V \to \R^d$ be a function with the property that $f(v_\sigma) = (x_i^{(\dim\sigma)}, 1)$ for $v_\sigma \in A_i^{(\dim \sigma)}$ and $f(v_\sigma) = (-x_i^{(\dim\sigma)}, -1)$ for $v_\sigma \in (-A_i^{(\dim \sigma)})$. If $v_\sigma$ is in multiple~$A_i$ then choose one arbitrarily, but in such a way that $f(-v_\sigma) = -f(v_\sigma)$. 

    We can think of $f$ as a map $f\colon \Sigma' \to \R^d$ by linearly extending it to the faces of~$\Sigma'$. By Lemma~\ref{lem:radon} the zeros of~$f$ then precisely correspond to the Radon pairs $(A,B)$ for~$\bigcup_j X^{(j)}$, where $A\cup B$ contains at most one point in each~$X^{(j)}$. Here we use that by the choice of~$\varepsilon$ if $f(v) = -f(w)$ then $v$ and~$w$ are not in a common face of~$\Sigma$. By the Borsuk--Ulam theorem, $f$ has a zero. Now let $\varepsilon$ go to zero and use compactness of~$S^d$. 
\end{proof}

\begin{rem}
    We add two comments about generalizations of Theorem~\ref{thm:fan-gen-col}:
    \begin{compactitem}[(1)]
        \item Analogous reasoning yields a colorful generalization of Theorem~\ref{thm:fan-gen-cont}.
        \item The same reasoning works without changes even if the condition $A_i^{(j)} \cap (-A_i^{(\ell)}) =\emptyset$ for all $i \in [m]$ and $j \ne \ell\in[d+1]$ is dropped. In this case $\pi$ is no longer necessarily injective. In this case, if $\bigcup_j X^{(j)}$ is assumed to be generic, we get a partition $S\sqcup T$ of~$[d+1]$ and points $x_{i_s}^{(s)} \in X^{(s)}$ for $s\in S$ and $x_{i_t}^{(t)} \in X^{(t)}$ for $t\in T$ such that 
        \[\conv \{x_{i_s}^{(s)} \, |\, s\in S\} \cap\conv \{x_{i_t}^{(t)} \, |\, t\in T\}\ne\emptyset  \ \text{and} \ \bigcap_{s \in S} A_{i_s}^{(s)} \cap \bigcap_{t \in T} (-A_{i_t}^{(t)}) \ne \emptyset.\]
    \end{compactitem}
\end{rem}

\section{Consequences and context: Topology}
\label{sec:consequences-top}

We collect some consequences of the main results proved in the preceding section. Here we focus on topological consequences, such as intersection combinatorics within sphere coverings and non-embeddability results.

\subsection{Fan's theorem and sphere coverings}
Fan's theorem, Theorem~\ref{thm:fan}, is the special case of Theorem~\ref{thm:fan-gen} that $X$ is placed along the moment curve $\gamma(t) = (t, t^2, \dots, t^{d-1})$ in~$\R^{d-1}$. This follows from Gale's evenness criterion~\cite{gale1963neighborly}: Two disjoint sets $A$ and $B$ of in total $d+1$ points along~$\gamma$ have intersecting convex hulls if and only if between any two points of~$A$ there is at least one point of~$B$ along~$\gamma$ and vice versa, and if $A$ and $B$ involve less than $d+1$ points then their convex hulls do not intersect. By using the colorful generalization, Theorem~\ref{thm:fan-gen-col}, instead, we obtain the following colorful Fan's theorem:

\begin{cor}
\label{cor:fan-col}
    Let $m$ be a positive integer, and let $A_1^{(j)}, \dots, A_m^{(j)} \subset S^d$ be closed sets for $j \in [d+1]$ such that $A_i^{(j)} \cap (-A_i^{(\ell)}) =\emptyset$ for all $i \in [m]$ and for all $j \ne \ell\in [d+1]$. Further assume that $\bigcup_i A_i^{(j)} \cup \bigcup_i (-A_i^{(j)}) = S^d$ for all $j \in [d+1]$. Then there are indices $i_1 < i_2 < \dots < i_{d+1}$ and a bijection~$\pi\colon\{i_1, \dots, i_{d+1}\} \to [d+1]$ such that $\bigcap_{k = 1}^{d+1} (-1)^k A_{i_k}^{(\pi(k))} \ne \emptyset$.
\end{cor}

The case $m=d+1$ of this is our earlier result~{\cite[Thm.~3.3]{FrickWellner2023}}. There this result is stated for an arbitrary, not necessarily alternating, choice of signs. This also follows from Theorem~\ref{thm:fan-gen-col} since for $d+1$ points in~$\R^{d-1}$ any partition of the points may be prescribed as a Radon pair. (The case of one empty part can be handled by increasing the dimension by one.) See Meunier and Su~\cite{MeunierSu2019} for a different but related colorful Fan's theorem; see~\cite{FrickWellner2023} for a discussion of the differences. 

All the ``usual'' corollaries of Fan's theorem, such as the Lusternik--Schnirelman theorem~\cite{lusternik1930topological}, Tucker's combinatorial lemma~\cite{Tucker1945}, the KKM theorem and Brouwer's fixed point theorem, are special cases of Theorem~\ref{thm:fan-gen} as well, and one can thus derive colorful generalization from Theorem~\ref{thm:fan-gen-col} for free. For the KKM theorem and fixed-point theorems this is discussed in~\cite{FrickWellner2023} and we comment on this more in Subsection~\ref{subsec:fixedpt}. Recently Chase, Chornomaz, Moran, and Yehudayoff~\cite{chase2024local} proved a ``local'' variant of the Lusternik--Schnirelman theorem and applied it to Kneser-type colorings and PAC learning results. One version of their result asserts that for any covering of $S^d$, $d\ge 1$, by closed sets $A_1, \dots, A_m$ with $A_i \cap (-A_i) = \emptyset$ for all~$i \in [m]$, at least~$\lceil \tfrac{d+3}{2} \rceil$ of the $A_i$ intersect. This result is also given in~\cite[Thm.~3.4]{simonyi2009local}, where it is attributed to~\cite{ izydorek1995antipodal, jaworowski2000periodic, zbMATH03478089}.
We will derive a colorful generalization from Corollary~\ref{cor:fan-col}. We find it instructive to first derive the following variant:

\begin{cor}
\label{cor:localLS-signs}
     Let $A_1^{(j)}, \dots, A_m^{(j)} \subset S^d$, $j\in [d+2]$, be closed sets such that $A_i^{(j)} \cap(-A_i^{(\ell)}) = \emptyset$ for all~$i \in [m]$ and all~$j \ne \ell \in [d+2]$. Further assume that $\bigcup_i A_i^{(j)} = S^d$ for all $j \in [d+2]$. Then there are disjoint sets $S, T \subset [m]$ of sizes~$\lceil\tfrac{d+2}{2} \rceil$ and~$\lfloor \tfrac{d+2}{2} \rfloor$, respectively, and a bijection $\pi\colon S\cup T \to [d+2]$ such that $\bigcap_{i\in S} A_i^{(\pi(i))} \cap \bigcap_{i\in T} (-A_i^{(\pi(i))}) \ne \emptyset$.
\end{cor}

\begin{proof}
    Embed $S^d$ into $S^{d+1}$ as an equator. We will refer to the two points in $S^{d+1}$ at maximal distance from~$S^d$ as the north and south pole. Define $B_i^{(j)} \subset S^{d+1}$ to be the set of all points on geodesics connecting a point in $A_i^{(j)}$ to the north pole. Since $\bigcup_i A_i^{(j)} = S^d$ for all $j \in [d+2]$ we have that $\bigcup_i B_i^{(j)} \cup \bigcup_i (-B_i^{(j)}) = S^{d+1}$ for all $j \in [d+2]$. Further $B_i^{(j)} \cap (-B_i^{(\ell)}) = \emptyset$ for all $i \in [m]$ and all $j \ne \ell \in [d+2]$. By Corollary~\ref{cor:fan-col} there are $i_1 < i_2 < \dots < i_{d+2}$ and a bijection~$\pi\colon\{i_1, \dots, i_{d+2}\} \to [d+2]$ such that $\bigcap_{k = 1}^{d+2} (-1)^k B_{i_k}^{(\pi(k))} \ne \emptyset$. As both plus and minus signs appear in this intersection, any such intersection point can only be on~$S^d$ and thus $\bigcap_{k = 1}^{d+2} (-1)^k A_{i_k}^{(\pi(k))} \ne \emptyset$.
\end{proof}

In particular, Corollary~\ref{cor:localLS-signs} implies that a covering of $S^d$ with closed sets that each do not contain antipodes must use at least $d+2$ sets. This is the Lusternik--Schnirelman theorem~\cite{lusternik1930topological}, and we can think of Corollary~\ref{cor:localLS-signs} as a colorful generalization of this result. The uncolored version of Corollary~\ref{cor:localLS-signs} for $m=d+2$, with the additional insight that in this case the partition $S \sqcup T$ may be arbitrarily prescribed, is called the Bacon--Tucker theorem~\cite{bacon1966equivalent, Tucker1945}. Similar arguments now give the colorful generalization of the result of Chase, Chornomaz, Moran, and Yehudayoff:

\begin{cor}
\label{cor:localLS}
     Let $A_1^{(j)}, \dots, A_m^{(j)} \subset S^d$, $j\in [d+3]$, be closed sets such that $A_i^{(j)} \cap(-A_i^{(\ell)}) = \emptyset$ for all~$i \in [m]$ and all~$j \ne \ell \in [d+3]$. Further assume that $\bigcup_i A_i^{(j)} = S^d$ for all $j \in [d+3]$. Then there is a set $S\subset [m]$ of size~$\lceil\tfrac{d+3}{2} \rceil$ and an injective function $\pi\colon S \to [d+3]$ such that $\bigcap_{i\in S} A_i^{(\pi(i))} \ne \emptyset$.
\end{cor}

\begin{proof}
    We now repeat a similar argument as in the proof of Corollary~\ref{cor:localLS-signs}. We construct the sets $B_i^{(j)} \subset S^{d+1}$ in the same way. Then embed $S^{d+1}$ into~$S^{d+2}$ as an equator determined by an orthogonal line that intersects~$S^{d+2}$ in the north and the south pole. Let $C_i^{(j)} \subset S^{d+2}$ be the set of points on geodesics from a point in $B_i^{(j)}$ to the north pole. Additionally, let $C_{m+i}^{(j)}$ be the set of points on geodesics from a point in $-B_i^{(j)}$ to the north pole. We still have that $C_i^{(j)} \cap (-C_i^{(\ell)}) = \emptyset$ for all $i \in [m]$ and all $j \ne \ell \in [d+3]$. Since $\bigcup_i B_i^{(j)} \cup \bigcup_i (-B_i^{(j)}) = S^{d+1}$ for all $j \in [d+3]$, we have that $\bigcup_i C_i^{(j)} \cup \bigcup_i (-C_i^{(j)}) = S^{d+2}$ for all $j \in [d+3]$. By Corollary~\ref{cor:fan-col} there are $i_1 < i_2 < \dots < i_{d+3}$ and a bijection~$\pi\colon\{i_1, \dots, i_{d+3}\} \to [d+3]$ such that $\bigcap_{k = 1}^{d+3} (-1)^k C_{i_k}^{(\pi(k))} \ne \emptyset$. Since both signs appear, any such intersection point must be on the equator~$S^{d+1}$, and so $\bigcap_{k = 1}^{d+3} (-1)^k \varepsilon_{i_k}B_{i_k}^{(\pi(k))} \ne \emptyset$, where $\varepsilon_{\ell} = 1$ if $\ell \le m$ and $\varepsilon_{\ell} = -1$ if $\ell > m$. Thus $(-1)^k\cdot \varepsilon_{i_k}$ must take both values $-1$ and~$+1$, and so by the same reasoning as before, any intersection actually occurs in $S^d$ and thus $\bigcap_{k = 1}^{d+3} (-1)^k \varepsilon_{i_k} A_{i_k}^{(\pi(k))} \ne \emptyset$. By potentially flipping the sign of all sets involved in the intersection, we may assume that at least $\lceil\tfrac{d+3}{2} \rceil$ of the signs are positive, finishing the proof.
\end{proof}

The bound of~$\lceil\tfrac{d+3}{2} \rceil$ cannot be improved as in shown in~\cite{chase2024local} for the non-colorful version.

\subsection{Radon-type intersection results and non-embeddability results}
\label{subsec:radon}
Let $X \subset \R^{d-1}$ be a set of $m$ points that we identify with~$[m]$. By exhibiting closed sets $A_1, \dots, A_m \subset S^d$ that together with their negatives $-A_i$ cover the sphere and with $A_i\cap(-A_i) =\emptyset$ for all~$i\in [m]$, we can constrain the combinatorics of Radon pairs $(S,T)$ for~$X$: Among all pairs of disjoint sets~$(S,T)$ with $\bigcap_{i\in S} A_i\cap \bigcap_{i\in T} (-A_i) \ne\emptyset$ there must be a Radon pair. Theorem~\ref{thm:fan-gen-cont} provides a continuous generalization, a non-embeddability result. Note that the construction of one appropriate antipodal sphere covering is sufficient to obstruct the existence of any embedding of some simplicial complex into~$\R^{d-1}$. We give some examples:

\begin{ex}
For $m=d+1$ there are closed sets $A_1, \dots, A_{d+1} \subset S^d$ that, together with their antipodal copies~$-A_i$, cover the entire sphere and satisfy $A_i \cap (-A_i) = \emptyset$; for example, $A_i = \{x\in S^d\, |\, x_i \ge \frac{1}{\sqrt{d+1}}\}$ is such a collection of sets. In particular, Theorem~\ref{thm:fan-gen} implies that there must be a Radon partition among any $m = d+1$ points in~$\R^{d-1}$. This is Radon's lemma~\cite{radon1921mengen}. Theorem~\ref{thm:fan-gen-cont} thus implies the topological Radon theorem~\cite{Bajmoczy1979} that for any continuous map $f\colon \Delta_d \to \R^{d-1}$ two disjoint faces have intersecting images. 
\end{ex}

\begin{ex}
Let $r_0$ be the distance from a vertex to the barycenter in a regular $(d+1)$-simplex of the same side length as a regular $(2d+2)$-simplex inscribed into the unit sphere~$S^{2d+1}$. Thus for $r < r_0$ no $d+2$ balls of radius~$r$ around the vertices of the inscribed $(2d+2)$-simplex intersect. Let $A_1, \dots, A_{2d+3}$ be the balls around the vertices of this simplex for some $r$ close to but less than~$r_0$. Then the $A_i$ and $-A_i$ cover~$S^{2d+1}$. Since no $d+2$ of the $A_i$ intersect, and the same is true for the~$-A_i$, Theorem~\ref{thm:fan-gen} shows that for any point set of size $2d+3$ in~$\R^{2d}$ the convex hulls of two disjoint sets of size $d+1$ intersect. This and more generally its continuous generalization, which follows from Theorem~\ref{thm:fan-gen-cont}, is van Kampen's theorem; see~\cite{van1933komplexe}.
\end{ex}

\begin{ex}
Flores' result~\cite{flores1933} that the $(d+1)$-fold join of a $3$-point space does not embed into~$\R^{2d}$ also follows from constructing an appropriate sphere covering and invoking Theorem~\ref{thm:fan-gen-cont}. There are three pairwise disjoint closed sets $A, B, C$ in the circle $S^1$ that together with $-A, -B, -C$ cover the entire circle. From this and since $S^{2d+1}$ is the join of $d+1$ circles, it is easy to construct $A_i, B_i, C_i \subset S^{2d+1}$ for $i \in [d+1]$ that together with their antipodal copies cover the entire sphere and such that $A_i, B_i, C_i$ are pairwise disjoint for every~$i \in [d+1]$. Flores' result now follows from Theorem~\ref{thm:fan-gen-cont}.
\end{ex}

\subsection{Odd maps without common zeros}
In his original work, Fan derives a corollary of Theorem~\ref{thm:fan} that for any $m$ odd maps~$f_1, \dots, f_m \colon S^d\to \R$ there is a point~$x_0 \in S^d$ and $i_1 < i_2 < \dots < i_{d+1}$ where $f_{i_j}$ have maximal absolute value among all~$f_i$, and such that their signs alternate. We can now generalize this beyond alternating signs:

\begin{cor}
\label{cor:odd}
    Let $m$ be a positive integer, and let $X = \{x_1, \dots, x_m\} \subset \R^{d-1}$ be a set of $m$ points. Let $f_1, \dots, f_m \colon S^d \to \R$ be odd maps without common zero. Then there are disjoint subsets $S, T \subset [m]$ and a point $x_0\in S^d$ such that $\conv\{x_i\, |\, i \in S\} \cap \conv\{x_i\, |\, i \in T\} \ne \emptyset$ and
    \[f_i(x_0) = \max_{j\in [m]} |f_j(x_0)| \ \text{for} \ i \in S \ \text{and} \
    f_i(x_0) = -\max_{j\in [m]} |f_j(x_0)| \ \text{for} \ i \in T .\]
\end{cor}

\begin{proof}
    Let $A_i = \{x\in S^d\, |\, f_i(x) = \max_{j\in [m]} |f_j(x_0)|\}$ for $i \in [m]$. The sets $A_i, -A_i$, $i \in [m]$ are closed and cover~$S^d$. Since the $f_i$ do not have a common zero, we have that $A_i \cap (-A_i) = \emptyset$. Now apply Theorem~\ref{thm:fan-gen}.
\end{proof}

\subsection{Fixed-point theorems}
\label{subsec:fixedpt}
The Borsuk--Ulam theorem implies Brouwer's fixed point theorem: Any continuous map~$f$ from the closed unit ball $B^d$ to itself has a fixed point, that is, there is an $x\in B^d$ with $f(x) = x$. In the same way that the Borsuk--Ulam theorem may be equivalently stated in terms of set coverings of the sphere as the Lusternik--Schnirelman theorem, Brouwer's fixed point theorem can equivalently be stated in terms of set coverings of the ball as the KKM theorem~\cite{KnasterKuratowskiMazurkiewicz1929}. It states that for any covering of the $(m-1)$-simplex~$\Delta_{m-1}$ with $m$ closed sets, one for each vertex, such that every face is covered by the sets corresponding to its vertices, the sets all have a point in common. Since our main results generalize the Borsuk--Ulam theorem, one may wonder how Brouwer's fixed point theorem, or equivalently the KKM theorem, is strengthened by using our results instead of the Borsuk--Ulam theorem in these implications. Here we point out that the analogous results extending the KKM theorem have been proven already.

The KKM theorem was generalized to the KKMS theorem by Shapley~\cite{Shapley1972} and even further to a result about polytope coverings by Komiya~\cite{Komiya1994}. For each subsequent extension a colorful variant can be found in the literature: The colorful KKM theorem of Gale~\cite{Gale1984}, the colorful KKMS theorem of Shih and Lee~\cite{shih1993}, and the colorful Komiya theorem of the first author and Zerbib~\cite{FrickZerbib2019}. 
We state Komiya's theorem here so that the reader may compare its statement with Theorem~\ref{thm:fan-gen}.

\begin{thm}[Komiya~\cite{Komiya1994}]
\label{thm:komiya}
     Let $P \subset \R^d$ be a $d$-dimensional polytope. For every face $\sigma$ of~$P$ let $y_\sigma \in \sigma$ be a point and let $A_\sigma \subset P$ be a closed set. Suppose that $\sigma \subset \bigcup_{\tau \subset \sigma} A_\tau$ for every face $\sigma$ of~$P$. Then there are faces $\sigma_1, \dots, \sigma_k$ of~$P$ such that
     \[
         y_P \in \conv\{y_{\sigma_1}, \dots, y_{\sigma_k}\} \ \text{and} \ \bigcap_{i \in [k]} A_{\sigma_i} \ne \emptyset.
     \]
\end{thm}

Theorem~\ref{thm:komiya} and its colorful extension follow from Theorems~\ref{thm:fan-gen} and~\ref{thm:fan-gen-col}, respectively. We omit this derivation as it is not any simpler than the short proofs that are already found in the literature. It is tempting to think that conversely perhaps Theorem~\ref{thm:komiya} easily implies Theorem~\ref{thm:fan-gen}. This is unlikely: Brouwer's fixed point theorem, or more precisely the search problem of finding an approximate fixed point, and its relatives belong to the complexity class PPAD, whereas the corresponding search problem for the Borsuk--Ulam theorem is in PPA.  It is believed that these are distinct complexity classes with the Borsuk--Ulam theorem and its relatives being strictly stronger statements whose search problems are outside of PPAD. Theorem~\ref{thm:fan-gen} is to Komiya's theorem (and Theorem~\ref{thm:fan-gen-col} to the colorful Komiya theorem) as the Borsuk--Ulam theorem is to Brouwer's fixed point theorem.

We record the following simple consequence of Theorem~\ref{thm:fan-gen} that can be seen as a $\Z/2$-equivariant strengthening of the KKM theorem:

\begin{cor}
\label{cor:fan-gen}
    Let $A_1, \dots, A_{d+1} \subset S^d$ be closed sets such that $A_i \cap (-A_i) =\emptyset$ for all $i \in [d+1]$ and $\bigcup_i A_i \cup \bigcup_i (-A_i) = S^d$. Then $\bigcap_i A_i \ne \emptyset$.
\end{cor}

\begin{proof}
    Let $B_i = A_i$ for $i \in [d]$, and let $B_{d+1} = -A_{d+1}$. Use Theorem~\ref{thm:fan-gen} on these sets and $X = \{x_1,\dots,x_{d+1}\} \subset \R^{d-1}$ constructed by placing $x_{d+1}$ in the interior of the simplex spanned by $x_1, \dots, x_d$. Thus 
    \[
    \emptyset \ne \bigcap_{i\in [d]} B_i \cap (-B_{d+1}) = \bigcap_{i \in [d+1]} A_i.
    \]
\end{proof}

\begin{rem}
    Corollary~\ref{cor:fan-gen} implies that in Theorem~\ref{thm:fan-gen} for $m=d+1$ any sign pattern can be prescribed for the intersection. Precisely: Let $A_1, \dots, A_{d+1} \subset S^d$ be closed sets such that $A_i \cap (-A_i) =\emptyset$ for all $i \in [d+1]$ and $\bigcup_i A_i \cup \bigcup_i (-A_i) = S^d$. Then for all partitions $S \sqcup T$ of~$[d+1]$ into two (not necessarily non-empty) sets we have that \[\bigcap_{i \in S} A_i \cap \bigcap_{i \in T} (-A_i) \ne \emptyset.\]
\end{rem}

In fact, Komiya's theorem and its colorful generalization were extended further to ``sparse'' and to matroidal version; see~\cite{ mcginnis2024matroid, mcginnis2024sparse, soberon2022fair}. It would be interesting to derive similar extension of Theorem~\ref{thm:fan-gen} and Theorem~\ref{thm:fan-gen-col}.

\section{Consequences and context: Combinatorics}
\label{sec:consequences-comb}

\subsection{Kneser colorings, rainbow faces, and the topological Hall theorem}
\label{subsec:kneser}
Since Lov\'asz's proof of Kneser's conjecture~\cite{Lov78} the Borsuk--Ulam theorem has been used to establish lower bounds for chromatic numbers of graphs; see~\cite{daneshpajouh2025box, matousek2004topological} for surveys. The \emph{chromatic number}~$\chi(G)$ of a graph~$G$ is the smallest number of colors required to color the vertices of~$G$ such that no edge has endpoints of the same color. All graphs we are considering will be finite and without loops. Topological lower bounds for~$\chi(G)$ are in terms of the topology of simplicial complexes associated to~$G$ that we explain now.

Let $G$ be a graph on vertex set~$V$ and let $N(G)$ be its \emph{neighborhood complex}, that is, the simplicial complex of $\sigma \subset V$ that are contained in a common neighborhood: There is a $v \in V$ such that $(v,w)$ is an edge for all~$w\in \sigma$. 
The \emph{box complex}~$B(G)$ of~$G$ is a simplicial complex on $V \times \{-1,+1\}$, where $(\sigma \times \{-1\}) \cup (\tau \times \{+1\})$ is a face of~$B(G)$ if $\sigma$ and $\tau$ are faces of $N(G)$ with the property that for all $v \in \sigma$ and all $w\in \tau$ there is an edge~$(v,w)$, that is, the complete bipartite graph between $\sigma$ and $\tau$ is a subgraph of~$G$. The complexes $N(G)$ and $B(G)$ are homotopy equivalent~\cite{matousek2004topological}. Lov\'asz showed that $\chi(G) \ge \mathrm{conn}(N(G))+3$, where $\mathrm{conn}\, X$ denotes the homotopical connectivity of~$X$. A reformulation, and slight strengthening, of this states that if there is a $\Z/2$-equivariant map $S^d \to B(G)$ then $\chi(G) \ge d+2$.

\begin{rem}
\label{rem:B0}
    There are various different but closely related versions of box complexes in the literature; see~\cite{daneshpajouh2025box, matousek2004topological}. We highlight one other variant: The complex $B_0(G)$ is obtained from $B(G)$ by adding $V \times \{+1\}$ and $V \times \{-1\}$ as faces (and all their respective subsets). Up to $\Z/2$-equivariant homotopy $B_0(G)$ is the suspension of~$B(G)$; see Csorba~\cite{csorba2007homotopy} and \v Zivaljevi\'c~\cite{vzivaljevic2005wi}. In particular, if there is a $\Z/2$-equivariant map $S^d \to B(G)$, then by suspending there is a $\Z/2$-equivariant map $S^{d+1} \to B_0(G)$.
\end{rem}

For graphs~$G$, where the Borsuk--Ulam theorem provides lower bounds for~$\chi(G)$, Fan's theorem and more generally Theorem~\ref{thm:fan-gen} will give insight into the structure of proper colorings, in addition to a bound on the number of colors that are required. For such applications of Fan's theorem see~\cite{fan1982evenly, hajiabolhassan2011generalization, SimonyiTardos2006, simonyi2007colorful}. 
We first need additional definitions before we can state two of these results. For a finite set system~$\mathcal F$, let $\mathrm{KG}(\mathcal F)$ denote its \emph{Kneser graph}, that is, the graph on vertex set~$\mathcal F$ with an edge $(A,B)$ whenever $A,B \in \mathcal F$ are disjoint. The \emph{colorability defect}~$\mathrm{cd}(\mathcal F)$ of a set system on~$[n]$ is the quantity
\[
    \mathrm{cd}(\mathcal F) = n - \max\{|A\cup B|\, |\, A,B\subset [n], \, \text{for all} \ F\in \mathcal F: F\not\subset A \ \text{and} \ F\not\subset B\}.
\]
In words, $\mathrm{cd}(\mathcal F)$ quantifies the size of the largest subset~$X$ of~$[n]$ such that the induced subhypergraph of~$\mathcal F$ on~$X$ admits a proper $2$-coloring. Dol'nikov~\cite{dol1988certain} showed that $\chi(\mathrm{KG}(\mathcal F)) \ge \mathrm{cd}(\mathcal F)$; this is because if $d \le \mathrm{cd}(\mathcal F)$ then there is a $\Z/2$-equivariant map $S^{d-2} \to B(\mathrm{KG}(\mathcal F))$.

\begin{thm}[Simonyi and Tardos~\cite{simonyi2007colorful}]
\label{thm:zigzag}
    Let $m\ge 1$ be an integer, and let $\mathcal F$ be a set system of non-empty sets with $\mathrm{cd}(\mathcal F) = c$. Then for any proper coloring of~$\mathrm{KG}(\mathcal F)$ with $m$ colors there is a complete bipartite subgraph $K_{\lceil c/2 \rceil, \lfloor c/2 \rfloor}$ such that $c$ different colors occur alternating with respect to their natural order on the two sides of the bipartite graph.
\end{thm}

In particular, this implies Dol'nikov's theorem that $m\ge c$. Simonyi and Tardos~\cite{simonyi2007colorful} further show that if there is a $\Z/2$-equivariant map $S^d \to B(\mathrm{KG}(\mathcal F))$ then for any proper coloring of $\mathrm{KG}(\mathcal F)$ with $d+1$ colors and any partition $A\sqcup B$ of $[d+1]$ into non-empty parts, there is a complete bipartite subgraph $K_{|A|,|B|}$ of~$\mathrm{KG}(\mathcal F)$ such that the colors in~$A$ appear on one side, and the colors in~$B$ on the other side. This generalizes a result of Csorba, Lange, Schurr, and Wassmer~\cite{csorba2004box} that such complete bipartite subgraphs (without any reference to colorings) must exist.

For a proper coloring~$c$ of a graph~$G$ with an arbitrary number of colors, let $\ell(c)$ denote the largest number of colors that appear in the neighborhood of a vertex~$v$ of~$G$. The \emph{local chromatic number}~$\psi(G)$ of a graph~$G$ is the minimum of $\ell(c)+1$ taken over all proper colorings~$c$ of~$G$. Another related result is the following:

\begin{thm}[Simonyi and Tardos~\cite{SimonyiTardos2006}]
\label{thm:local-chr}
    Let $G$ be a graph such that there is a $\Z/2$-equivariant map $S^d \to B_0(G)$. Then $\psi(G) \ge \lceil \frac{d+1}{2} \rceil +1$.
\end{thm}

Later, Simonyi, Tardos, and Vre\'cica~\cite{simonyi2009local} proved a related result that is often stronger; see Theorem~\ref{thm:local-chr2}.

Theorems~\ref{thm:zigzag} and~\ref{thm:local-chr} give structural insight into proper colorings of graphs whose chromatic number is governed by the Borsuk--Ulam theorem. Theorem~\ref{thm:fan-gen} shows that such structure is governed by intersection patterns of convex hulls in Euclidean space. Indeed, Simonyi and Tardos show that Theorem~\ref{thm:zigzag} is a consequence of Fan's theorem, that is, Theorem~\ref{thm:fan-gen} with points in cyclic position. Our results exhibit that Theorem~\ref{thm:local-chr} is a consequence of van Kampen's theorem that for $d+2$ points in~$\R^{d-1}$ there is a partition of the points into two sets $A$ and $B$ of (almost) equal size such that $\conv A \cap \conv B \ne \emptyset$. More generally, we will show that for graphs $G$ where a $\Z/2$-equivariant map $S^d \to B(G)$ exists, the structure of proper colorings is governed by Radon-type intersection results.

In extending such structural results for proper colorings of graphs now using Theorem~\ref{thm:fan-gen} instead of Theorem~\ref{thm:fan}, we derive a result that asserts the existence of colorful substructures in sufficient generality to also encompass other results about rainbow substructures, most notably the topological Hall theorem, which we state now (see~{\cite[Thm.~2.1]{aharoni2022fractionally}}):

\begin{thm}[Aharoni, Haxell~\cite{aharoni2000hall}, Meshulam~\cite{meshulam2003domination}]
\label{thm:top-hall}
    Let $\Sigma$ be a simplicial complex with vertex labeling $f\colon V \to [d]$ such that for every $A \subset [d]$ the induced subcomplex of~$\Sigma$ on vertex set~$f^{-1}(A)$ is $(|A|-2)$-connected. Then there is a face~$\sigma$ of~$\Sigma$ such that $f(\sigma) = [d]$.
\end{thm} 

We explain the relation to Hall's matching theorem and more generally Hall-type matching results in hypergraphs in the next subsection. For now we aim to show that the aforementioned results are consequences of Theorem~\ref{thm:fan-gen}. In particular, there are generalizations of Theorem~\ref{thm:top-hall} that give structural insights into vertex labelings with more than $d+1$ colors in analogy with Theorem~\ref{thm:zigzag}.

Let $\Sigma$ be a simplicial complex on vertex set~$V$. Let $f\colon V \to [m]$ be a labeling of the vertices. The \emph{labeling complex}~$\Lambda(\Sigma, f)$ is the simplicial complex on $V \times \{-1,+1\}$ with faces of the form $(\sigma \times \{-1\}) \cup (\tau \times \{+1\})$, where $\sigma$ and $\tau$ are faces of~$\Sigma$ such that $f(\sigma) \cap f(\tau) = \emptyset$. The complex~$\Lambda(\Sigma, f)$ has a free $\Z/2$-action that swaps $(v,-1)$ and~$(v,+1)$. A face $\sigma$ of~$\Sigma$ is \emph{colorful} if $f(\sigma) = [m]$. If moreover $\sigma$ has exactly $m$ elements, so that every label appears exactly once on~$\sigma$, we call $\sigma$ \emph{rainbow}. Any colorful face has a rainbow subface.

\begin{ex}
\label{ex:box}
    Let $G$ be a graph with neighborhood complex~$N(G)$, and let $f\colon V \to [c]$ be a proper coloring of~$G$. A proper coloring cannot use all colors on a neighborhood of a vertex, and so $N(G)$ does not have rainbow faces. Let $(\sigma \times \{-1\}) \cup (\tau \times \{+1\})$ be a face of~$B(G)$. Since $f$ is a proper coloring and $G$ contains the complete bipartite graph on $\sigma$ and~$\tau$, we have that $f(\sigma) \cap f(\tau) = \emptyset$. Thus $(\sigma \times \{-1\}) \cup (\tau \times \{+1\})$ is a face of $\Lambda(N(G), f)$, and so $B(G) \subset \Lambda(N(G), f)$. The complex $\Lambda(N(G), f)$ can be properly larger than the box complex~$B(G)$: Two disjoint subsets of vertices $\sigma, \tau \subset V$, each contained in a common neighborhood, that have disjoint sets of colors will induce a face in $\Lambda(N(G), f)$, independent of which edges between $\sigma$ and $\tau$ are present in~$G$.
\end{ex}

This example and Remark~\ref{rem:B0} motivate the following definition: For a simplicial complex~$\Sigma$ on~$V$ and labeling $f\colon V \to [m]$, let $\Lambda_0(\Sigma,f)$ be the simplicial complex obtained from~$\Lambda(\Sigma, f)$ by adding all subsets of $V \times \{+1\}$ and all subsets of $V \times \{-1\}$ as faces. By Example~\ref{ex:box}, for $G$ a graph and $f$ a proper coloring of~$G$ we have that $B_0(G) \subset \Lambda_0(N(G),f)$.
The following corollary exhibits ``colorful substructures'' and generalizes the results mentioned above:

\begin{cor}
\label{cor:rainbow}
     Let $m$ be a positive integer, and let $X = \{x_1, \dots, x_m\} \subset \R^{d-1}$ be a set of $m$ points. 
     \begin{compactenum}[(i)]
        \item Let $\Sigma$ be a simplicial complex with vertex labeling $f\colon V \to [m]$ such that there is a $\Z/2$-equivariant map $S^d \to \Lambda_0(\Sigma,f)$. Then there are disjoint subsets $S, T \subset [m]$ and disjoint faces~$\sigma$ and~$\tau$ of~$\Sigma$ such that \[\conv\{x_i\, |\, i \in S\} \cap \conv\{x_i\, |\, i \in T\} \ne \emptyset, \ f(\sigma) = S, \ \text{and} \ f(\tau) = T.\]
        \item Let $G$ be a graph on vertex set~$V$ such that there is a $\Z/2$-equivariant map $S^{d} \to B_0(G)$. Then for $f\colon V\to [m]$ there are disjoint subsets $S, T \subset [m]$ and disjoint sets $\sigma, \tau \subset V$ with $(u,w)$ is an edge of~$G$ for every $u \in \sigma$ and every $w\in \tau$ such that \[\conv\{x_i\, |\, i \in S\} \cap \conv\{x_i\, |\, i \in T\} \ne \emptyset, \ f(\sigma) = S, \ \text{and} \ f(\tau) = T.\]
    
        \item Let $\Sigma$ be a simplicial complex with vertex labeling $f\colon V \to [d]$ such that there is a $\Z/2$-equivariant map $S^{d-1} \to \Lambda_0(\Sigma,f)$. Then there is a face~$\sigma$ of~$\Sigma$ such that $f(\sigma) = [d]$. 
    \end{compactenum}
\end{cor}

\begin{proof}
\begin{compactenum}[(i)]
    \item Let $\widehat f\colon V \times \{-1,+1\} \to \R^{d-1}$ be defined by $\widehat f(v, \pm 1) = x_{f(v)}$. Apply Lemma~\ref{lem:labeled-facet} to the simplicial complex~$\Lambda_0(\Sigma, f)$ and the map~$\widehat f$. Then there is a face $(\sigma \times \{+1\}) \cup (\tau \times \{-1\})$ of~$\Lambda_0(\Sigma, f)$ with $\conv \widehat f(\sigma \times \{+1\}) \cap \conv \widehat f(\tau \times \{-1\}) \ne \emptyset$. In particular, since both $\sigma$ and $\tau$ are non-empty, they are disjoint faces of~$\Sigma$ by the definition of~$\Lambda_0(\Sigma, f)$. For $S = f(\sigma)$ and $T = f(\tau)$ we have that $\conv\{x_i \,|\, i \in S\} \cap \conv\{x_i\,|\, i \in T\} \ne \emptyset$.
    \item Apply the reasoning above to $B_0(G) \subset \Lambda_0(N(G), f)$ with the difference that now $(\sigma \times \{+1\} \cup (\tau \times \{-1\})$ will be a face of~$B_0(G)$. Then $\sigma$ and $\tau$ induce a complete bipartite subgraph of~$G$ by definition of~$B_0(G)$.
    \item Let $v_0$ be a new vertex. Let $\widehat V = V \cup \{v_0\}$, and let $\widehat f \colon \widehat V \to [d+1]$ extend~$f$ by setting $\widehat f(v_0) = d+1$. Let $\widehat \Sigma$ be the cone over $\Sigma$, that is, $\sigma \cup \{v_0\}$ is a face of~$\widehat \Sigma$ for every face $\sigma$ of~$\Sigma$. Then $\Lambda_0(\widehat \Sigma, \widehat f)$ is the suspension of~$\Lambda_0(\Sigma, f)$ and thus there is a $\Z/2$-equivariant map $S^d \to \Lambda_0(\widehat \Sigma, \widehat f)$. Let $X = \{x_1, \dots, x_{d+1}\} \subset \R^{d-1}$ by a set with only Radon partition $(\{x_1, \dots, x_d\}, \{x_{d+1}\})$, such as an affinely independent set and its barycenter. By part~(i) there are disjoint faces $\sigma$ and $\tau$ of $\widehat \Sigma$ with $\widehat f(\sigma) = [d]$ and $\widehat f(\tau) = \{d+1\}$; in particular, $f(\sigma) = [d]$ and $\sigma$ is a face of~$\Sigma$.
\end{compactenum}
\end{proof}

Corollary~\ref{cor:rainbow} is a common generalization of the topological Hall theorem (Theorem~\ref{thm:top-hall}) and topological lower bounds for the chromatic number. Indeed, let $G$ be a graph such that there is a $\Z/2$-equivariant map $S^d \to B_0(G)$, and suppose $f\colon V \to [d]$ were a proper coloring of~$G$. Since $B_0(G) \subset \Lambda_0(N(G),f)$, Corollary~\ref{cor:rainbow} asserts that there is a neighborhood $\sigma$ in~$G$ with $f(\sigma) = [d]$, but then $f$ could not have been a proper coloring and thus $\chi(G) \ge d+1$. We will explain how to derive the topological Hall theorem below. First we remark that Corollary~\ref{cor:rainbow}(ii) generalizes Theorems~\ref{thm:zigzag} and~\ref{thm:local-chr}.

\begin{rem}
    If $\mathrm{cd}(\mathcal F) = c$ then there is a $\Z/2$-equivariant $S^{c-1} \to B_0(\mathrm{KG}(\mathcal F))$. Let $f\colon \mathcal F \to [m]$ be a proper $m$-coloring of~$\mathrm{KG}(\mathcal F)$. Place $x_1, \dots, x_m \in \R^{c-2}$ in cyclic position. Then Corollary~\ref{cor:rainbow}(ii) specializes to Theorem~\ref{thm:zigzag}, since points in cyclic position in~$\R^{c-2}$ have minimal Radon partitions $(A,B)$ of size~$c$, such that between any two points of~$A$ there is a point of~$B$ and vice versa. 
    The larger part, say~$A$, has size $\lceil \frac{c}{2} \rceil$, which for $d=c-1$ is $\lceil \frac{d+1}{2} \rceil$. Thus if there is a $\Z/2$-equivariant map $S^d \to B_0(G)$ then Corollary~\ref{cor:rainbow} implies that $\psi(G) \ge \lceil\frac{d+1}{2}\rceil+1$, recovering Theorem~\ref{thm:local-chr}. See also Theorem~\ref{thm:local-chr2}, which gives stronger bounds in certain cases.
\end{rem}

While Corollary~\ref{cor:rainbow} is optimal, in the sense that the same result for $X \subset \R^d$ is no longer true (for example, because it would imply a lower bound of~$d+2$ for the chromatic number), consequences for rainbow bipartite subgraphs and the local chromatic number can sometimes be improved by one. This is essentially due to the phenomenon of Corollary~\ref{cor:localLS}, that in any finite closed covering of $S^d$ with sets that do not contain antipodes, some $\lceil \frac{d+3}{2} \rceil$ sets intersect. We concretely state the following result of Simonyi, Tardos, and Vre\'cica that sometimes improves on the bound on local chromatic number of Theorem~\ref{thm:local-chr} by one:

\begin{thm}[Simonyi, Tardos, and Vre\'cica~\cite{simonyi2009local}]
\label{thm:local-chr2}
    Let $G$ be a graph such that there is a $\Z/2$-equivariant map $S^d \to B(G)$. Then $\psi(G) \ge \lfloor \frac{d}{2} \rfloor +3$.
\end{thm}

The existence of $\Z/2$-equivariant map $S^d \to B(G)$ implies that there is a $\Z/2$-equivariant map $S^{d+1} \to B_0(G)$, but the converse does not hold; see~\cite{simonyi2009local}. By Theorem~\ref{thm:local-chr}, this yields the bound $\psi(G) \ge \lceil \frac{d+2}{2} \rceil +1$, which depending on the parity of $d$ either agrees with the bound of Theorem~\ref{thm:local-chr2} or is worse by one.
The proof in~\cite{simonyi2009local} relies on the non-colorful version of Corollary~\ref{cor:localLS}.

We will now explain how to derive the topological Hall theorem, Theorem~\ref{thm:top-hall}, from the results above.

For a simplicial complex~$\Sigma$ with vertex labeling $f\colon V \to [d+1]$, we can think of $\Lambda(\Sigma,f)$ as labeled with $\{\pm 1, \dots, \pm (d+1)\}$: The vertex labeling $\widetilde f\colon V \times \{-1,+1\} \to \{\pm 1, \dots, \pm (d+1)\}$ is defined by $\widetilde f(v,\pm 1) = \pm f(v)$. Denote the induced subcomplex on a subset of vertices $W \subset V$ by~$\Sigma[W]$. Let $\widetilde A \subset \{\pm 1, \dots, \pm (d+1)\}$ be a set that whenever $j \in \widetilde A$ then $-j \notin \widetilde A$. Let $A = \{|a|\, |\, a\in \widetilde A\}$ be the set $\widetilde A$ with all signs changed to positive. It is easily verified that $\Sigma[f^{-1}(A)]$ is isomorphic to~$\Lambda(\Sigma,f)[\widetilde f^{-1}(\widetilde A)]$. 

\begin{lem}
\label{lem:conn-eq}
     Let $\Sigma$ be a simplicial complex with vertex labeling $f\colon V \to [d+1]$ such that for every $A \subset [d+1]$ the induced subcomplex of~$\Sigma$ on vertex set~$f^{-1}(A)$ is $(|A|-2)$-connected. Then there is a $\Z/2$-equivariant map $S^d \to \Lambda(\Sigma,f)$.
\end{lem}

\begin{proof}
    Identify the vertex set of~$(\Delta_d)^{*2}_\Delta$ with $\{\pm 1, \dots, \pm (d+1)\}$ in the natural way. Construct a $\Z/2$-equivariant map $h\colon (\Delta_d)^{*2}_\Delta \to \Lambda(\Sigma,f)$ skeleton-by-skeleton with the additional property that $h(A) \subset \Lambda(\Sigma,f)[\widetilde f^{-1}(\widetilde A)]$ for any face $\widetilde A$ of $(\Delta_d)^{*2}_\Delta$ (with notation as above). Suppose that $h$ has been defined on all faces of dimension at most~$k-1$, and let $\widetilde\sigma$ be a $k$-dimensional face of~$(\Delta_d)^{*2}_\Delta$. Since $\widetilde\sigma$ has $k+1$ vertices, $\Lambda(\Sigma,f)[\widetilde f^{-1}(\widetilde\sigma)] \cong \Sigma[f^{-1}(\sigma)]$ is $(k-1)$-connected, where $\sigma$, as before, contains the elements of~$\widetilde\sigma$ with all signs flipped to positive. Since $\partial\sigma$ is a $(k-1)$-sphere there is no obstruction to extending $h$ to $\widetilde\sigma$, and by symmetry to $-\widetilde\sigma$. Extending for every antipodal pair of $k$-faces, defines $h$ on all $k$-faces. By induction obtain a $\Z/2$-equivariant map $h\colon S^d \cong (\Delta_d)^{*2}_\Delta \to \Lambda(\Sigma,f)$.
\end{proof}

\begin{proof}[Proof of Theorem~\ref{thm:top-hall}]
Combine Lemma~\ref{lem:conn-eq} with Corollary~\ref{cor:rainbow}(iii).
\end{proof}

Since through Lemma~\ref{lem:conn-eq} Corollary~\ref{cor:rainbow}(iii) immediately implies the topological Hall theorem, Corollary~\ref{cor:rainbow}(i) is a ``Simonyi--Tardos''-style strengthening of the topological Hall theorem. Corollary~\ref{cor:rainbow}(i) gives structural information about pairs of faces that together exhibit $d+1$ colors. Related quantitative versions of the topological Hall theorem are due Meunier and Montejano~\cite{meunier2020different}.

\subsection{Hall-type results for hypergraphs}
Recall that Hall's matching theorem gives a necessary and sufficient condition for the existence of a matching in a bipartite graph that fully covers all vertices in one of the parts. A \emph{matching} is a set of edges that pairwise do not share any vertices. If $G$ is a bipartite graph on vertex set $X\sqcup Y$ such that for every edge $e \in E$ one endpoint is in~$X$ and one endpoint is in~$Y$, then there is a matching $M \subset E$ incident to every vertex in~$X$ if and only if $|N(A)| \ge |A|$ for all $A \subset X$. 

Let $\mathcal F$ be a family of sets, and let $f\colon \mathcal F \to [m]$ be a labeling of the sets. A \emph{system of disjoint representatives} (or \emph{transversal matching}) is a choice of $A_i \in f^{-1}(i)$ for every $i \in [m]$ such that the $A_i$ are pairwise disjoint. By $M(\mathcal F)$ we will denote the \emph{matching complex} of~$\mathcal F$, that is, the simplicial complex on vertex set~$\mathcal F$, where $\{B_1, \dots, B_\ell\}$ is a face if the $B_i$ are pairwise disjoint. Thus a rainbow face of~$M(\mathcal F)$ with respect to the labeling~$f$ is a system of disjoint representatives. Equivalently, we may think of this as $m$ families of sets $\mathcal F_1, \dots, \mathcal F_m$, where a system of disjoint representatives consists of pairwise disjoint $A_1 \in \mathcal F_1, \dots, A_m \in \mathcal F_m$.

For $G$ a bipartite graph on $X \sqcup Y$ let $\mathcal F_x$ be the set of neighbors of $x \in X$. A system of disjoint representatives $\{\{y_x\} \, |\, x\in X\}$ corresponds to a matching $\{(x,y_x)\, |\, x\in X\}$ of~$G$ that uses every vertex in~$X$. The disjoint representatives are singletons in this case. The topological Hall theorem for $M(\bigcup_x \mathcal F_x)$ implies Hall's matching theorem since the join of $k$ complexes of dimension~$0$ is $(k-2)$-connected. Similarly, the topological Hall theorem implies existence results for systems of disjoint representatives that have more than one element and can be viewed as hypergraph Hall theorems. We mention one instance that can be phrased in purely combinatorial terms.

Let $\mathcal F$ be a family of sets. The \emph{width}~$w(\mathcal F)$ is the minimal~$t$ for which there are $A_1, \dots, A_t \in \mathcal F$ such that for all $A \in \mathcal F$ there is an $i \in [t]$ with $A \cap A_i \ne \emptyset$. In particular, for any $k < w(\mathcal F)$ and for all $A_1, \dots, A_k \in \mathcal F$ there is an $A \in \mathcal F$ that is disjoint from all~$A_i$.

\begin{thm}[Aharoni and Haxell~\cite{aharoni2000hall}]
\label{thm:hall-width}
    Let $\mathcal F_1, \dots, \mathcal F_d$ be families of sets with $w(\bigcup_{i\in I} \mathcal F_i) \ge 2|I|-1$ for all $I \subset [d]$. Then there is a system of disjoint representatives $A_1 \in \mathcal F_1, \dots, A_d\in \mathcal F_d$.
\end{thm}

A simplicial complex is a \emph{flag complex} if every inclusion-minimal non-face has two elements. The relation between the combinatorial parameter width and the topology of induced subcomplexes of the matching complex is given by the following lemma (see~\cite{aharoni2000hall}), which combined with the topological Hall theorem yields Theorem~\ref{thm:hall-width}:

\begin{lem}
     Let $k\ge 2$ be an integer. Let $\Sigma$ be a flag complex such that every $2k-2$ vertices of~$\Sigma$ have a common neighbor. Then $\Sigma$ is $(k-2)$-connected. In particular, if $\mathcal F$ is a family of sets with $w(\mathcal F) \ge 2k-1$ then $M(\mathcal F)$ is $(k-2)$-connected.
\end{lem}

Corollary~\ref{cor:rainbow}(i) immediately implies the following generalization to more than $d$ labels:

\begin{cor}
\label{cor:hall-zigzag}
    Let $\mathcal F$ be a family of sets, and let $f\colon \mathcal F \to [m]$ be a labeling such that there is a $\Z/2$-equivariant map $h\colon S^{d-1} \to \Lambda_0(M(\mathcal F),f)$. Let $X = \{x_1, \dots, x_m\} \subset \R^{d-2}$. Then there are disjoint subsets $S, T \subset [m]$, and $A_i \in f^{-1}(i)$ for every $i\in S\cup T$ such that both $\{A_i\, |\, i \in S\}$ and $\{A_i\, |\, i \in T\}$ are systems of disjoint representatives and \[\conv\{x_i\, |\, i \in S\} \cap \conv\{x_i\, |\, i \in T\} \ne \emptyset.\]
\end{cor}

\section{Consequences and context: Discrete and convex geometry}
\label{sec:consequences-geom}

\subsection{Generalizations of the ham sandwich theorem}

A \emph{mass} is a compactly supported Borel measure~$\mu$ on~$\R^d$ with $\mu(\R^d)=1$ that vanishes on any \emph{hyperplane} $H = \{x \in \R^d\, |\, \langle x,z \rangle =b\}$ for $z\in \R^d \setminus \{0\}$ and $b\in \R$. The hyperplane~$H$ determines two halfspaces 
\[
    H^+ = \{x \in \R^d\, |\, \langle x,z \rangle \ge b\} \ \text{and} \ H^-= \{x \in \R^d\, |\, \langle x,z \rangle \le b\}.
\]
The definition of~$H^+$ and~$H^-$ works without change even if $z=0$, which corresponds to adding the two degenerate partitions into halfspaces $(\R^d, \emptyset)$ and~$(\emptyset, \R^d)$. Thus any $(z,b) \in \R^d \times \R$ corresponds to a partition into halfspaces, and by scaling it is sufficient to consider $(z,b) \in S^d \subset \R^d \times \R$. It is a standard fact that for a mass~$\mu$ the map 
\[
    S^d \to \R, \ (z,b) \mapsto \mu(\{x \in \R^d\, |\, \langle x,z \rangle \ge b\})
\]
is continuous; see~{\cite[p.~48]{Matousek2003book}}.

The ham sandwich theorem asserts that for $d$ masses $\mu_1, \dots, \mu_d$ on~$\R^d$ there is a hyperplane $H$ that simultaneously bisects all~$\mu_i$, that is, $\mu_i(H^+) = \mu_i(H^-)$ for all~$i \in [d]$. This is a standard consequence of the Borsuk--Ulam theorem, and thus by using Theorem~\ref{thm:fan-gen} instead we can prove the following generalization that additionally gives insight into the structure of bipartitions into halfspaces for more than~$d$ masses:

\begin{cor}
\label{cor:fan-ham}
    Let $m$ be a positive integer, and let $X = \{x_1, \dots, x_m\} \subset \R^{d-1}$ be a set of $m$ points. Let $\mu_1, \dots, \mu_m$ be masses on~$\R^d$, and suppose that there is a $c>0$ such that for every affine hyperplane~$H$ there is a $\ell\in [m]$ such that $|\mu_\ell(H^+)-\mu_\ell(H^-)|\ge c$. Then there is an affine hyperplane~$H$ and disjoint subsets $S, T \subset [m]$ such that $\conv\{x_i\, |\, i\in S\} \cap \conv\{x_i\, |\, i \in T\}\ne \emptyset$ and
    \[
    \mu_i(H^+)-\mu_i(H^-) \ge c \ \text{for all} \ i \in S \ \text{and} \ \mu_i(H^-)-\mu_i(H^+) \ge c \ \text{for all} \ i \in T.
    \]
\end{cor}

\begin{proof}
    To $(z,b) \in S^d \subset \R^d \times \R$ associate the hyperplane $H(z,b) = \{x \in \R^d\, |\, \langle x,z \rangle = b\}$. Let $A_i = \{(z,b) \in S^d\, |\, \mu_i(H(z,b)^+) - \mu_i(H(z,b)^-) \ge c\}$. These sets are closed since $(z,b) \mapsto \mu_i(H(z,b)^+)$ is continuous, and we have that $S^d = \bigcup_i A_i \cup \bigcup_i (-A_i)$. Now use Theorem~\ref{thm:fan-gen}.
\end{proof}

\begin{ex}
    For $m=d$ points in generic position in~$\R^{d-1}$, there is no Radon partition, and so Corollary~\ref{cor:fan-ham} implies the Ham Sandwich theorem.
    
    Let $\mu_1, \dots, \mu_{d+1}$ be masses on~$\R^d$ with $\mu_i(\R^d)=1$. If no hyperplane intersects the supports of all~$\mu_i$, then for every hyperplane some $\mu_i$ is entirely to one side of it, and we can choose $c=1$. Corollary~\ref{cor:fan-ham} then asserts that for any $I \subset [d+1]$ there is a hyperplane that leaves~$\mu_i$, $i \in I$, to the positive side, and the other~$\mu_i$, $i\notin I$, to the negative side. That is, the supports of the~$\mu_i$ are well-separated.

    Similarly, for $\mu_1, \dots, \mu_{d+2}$ masses on~$\R^d$ with $\mu_i(\R^d)=1$ that are not pierced by a single  hyperplane. There is a hyperplane that leaves any $\lceil\tfrac{d+2}{2}\rceil$ masses to the positive side, and the others to the negative side.
\end{ex}

To summarize, there are $d$ points in~$\R^{d-1}$ that do not admit a Radon partition, which implies the Ham Sandwich theorem; $d+1$ points may be placed in~$\R^{d-1}$ with a unique Radon partition and any partition may be realized, which implies that $d+1$ masses in~$\R^d$ that are not pierced by a single hyperplane are well-separated; $d+2$ points may be placed in~$\R^{d-1}$ such that the convex hulls of two prescribed disjoint subsets of sizes~$\lceil\tfrac{d+2}{2}\rceil$ and~$\lfloor\tfrac{d+2}{2}\rfloor$ intersect, 
which implies that for $d+2$ masses in~$\R^d$ that are not pierced by a single hyperplane there is a hyperplane that cuts off any subset of half the masses. In general, the intersection combinatorics of convex hulls of points in~$\R^{d-1}$ determine the convex geometry of masses in~$\R^d$.

By using Theorem~\ref{thm:fan-gen-col} instead of Theorem~\ref{thm:fan-gen} we get the following colorful generalization of the ham sandwich theorem. This extends our earlier result~\cite{FrickWellner2023}, which is the case $m=d+1$. 

\begin{cor}
\label{cor:fan-ham-col}
    Let $m$ be a positive integer, and let $X = \{x_1, \dots, x_m\} \subset \R^{d-1}$ be a set of $m$ points. Let $\mu_1^{(j)}, \dots, \mu_m^{(j)}$ be masses on~$\R^d$ for every $j\in[d+1]$, and suppose that there is a $c>0$ such that for every affine hyperplane~$H$ and every $j\in [d+1]$ there is a $\ell\in [m]$ such that $\mu_\ell^{(j)}(H^+)-\mu_\ell^{(j)}(H^-)\ge c$ or there is a $k\in [m]$ such that $\mu_k^{(j)}(H^-)-\mu_k^{(j)}(H^+)\ge c$ but not both. Then there is an affine hyperplane~$H$ and disjoint subsets $S, T \subset [m]$ along with an injective map $\pi\colon S\cup T\to [d+1]$ such that $\conv\{x_i\, |\, i\in S\} \cap \conv\{x_i\, |\, i \in T\}\ne \emptyset$ and
    \[
    \mu_i^{(\pi(i))}(H^+)-\mu_i^{(\pi(i))}(H^-) \ge c \ \text{for all} \ i \in S \ \text{and} \ \mu_i^{(\pi(i))}(H^-)-\mu_i^{(\pi(i))}(H^+) \ge c \ \text{for all} \ i \in T.
    \]
\end{cor}

\section{A generalization of Fan's theorem to a product of spheres}
\label{sec:two-spheres}

The proof strategy employed in this work can be used beyond $\Z/2$-symmetry. Here we develop a covering-labeling generalization of the non-existence of certain $(\Z/2)^2$-equivariant maps $S^d \times S^{d-1} \to S^{2d-2}$. Denote the standard generators of~$(\Z/2)^2$ by~$g_1$ and~$g_2$. The action on the domain is given by: For $(x,y) \in S^d \times S^{d-1}$ let $g_1\cdot(x,y) = (-x,y)$ and $g_2\cdot(x,y) = (x,-y)$. On the codomain both generators act non-trivially: $g_i\cdot z = -z$ for $z\in S^{2d-2}$. The non-existence result we will generalize is the following result of Ramos~\cite{Ramos1996}; see~\cite{MANILEVITSKA2006, ChanChenFrickHull2019} for subsequent proofs.

\begin{thm}[Ramos~\cite{Ramos1996}]
\label{thm:ramos}
    Let $d=2^t$ be a power of two. Then there is no $(\Z/2)^2$-equivariant map $f\colon S^d \times S^{d-1} \to S^{2d-2}$.
\end{thm}

Equivalently, any $(\Z/2)^2$-equivariant map $S^d \times S^{d-1} \to \R^{2d-1}$ has a zero for $d$ a power of two. Here both generators act by $g_i\cdot z = -z$ on~$\R^{2d-1}$. 

For $A\subset S^d \times S^{d-1}$ denote by $G\cdot A$ the $(\Z/2)^2$-orbit of~$A$, that is,
\[
    G\cdot A = A \cup g_1\cdot A \cup g_2 \cdot A \cup g_1g_2\cdot A.
\]

\begin{thm}
\label{thm:fan-prod}
    Let $d=2^t$ be a power of two. Let $m$ be a positive integer, and let $X = \{x_1, \dots, x_m\} \subset \R^{2d-2}$ be a set of $m$ points. Let $A_1, \dots, A_m \subset S^d \times S^{d-1}$ be closed sets such that $A_i \cap g_1\cdot A_i =\emptyset$, $A_i \cap g_2 \cdot A_i = \emptyset$ for all $i \in [m]$, and $\bigcup_i G\cdot A_i = S^d \times S^{d-1}$. Then there are disjoint subsets $S, T \subset [m]$ and functions $\alpha\colon S \to \{g_1, g_2\}$ and $\beta \colon T\to \{1, g_1g_2\}$ such that \[\conv\{x_i\, |\, i \in S\} \cap \conv\{x_i\, |\, i \in T\} \ne \emptyset \ \text{and} \ \bigcap_{i \in S} (\alpha(i)\cdot A_i) \cap \bigcap_{i \in T} (\beta(i)\cdot A_i) \ne \emptyset.\]
\end{thm}

\begin{proof}
    Let $\varepsilon > 0$ such that any $A_i$ is at distance at least~$\varepsilon$ from~$g_1\cdot A_i$ and~$g_2\cdot A_i$. Let $\Sigma$ be a $(\Z/2)^2$-symmetric triangulation of~$S^d \times S^{d-1}$, where every facet has diameter less than~$\varepsilon$. Let $V$ denote the vertex set of~$\Sigma$ and let $f \colon V \to \R^{2d-1}$ be a function with the property that $f(v) = (x_i, 1)$ for $v \in (A_i \cup g_1 g_2\cdot A_i)$ and $f(v) = (-x_i, -1)$ for $v \in (g_1\cdot A_i\cup g_2\cdot A_i)$. If $v$ is in multiple~$A_i$ then choose one arbitrarily, but in such a way that $f$ is $(\Z/2)^2$-equivariant. 

    We can think of $f$ as a map $f\colon \Sigma \to \R^{2d-1}$ by linearly extending it to the faces of~$\Sigma$. The zeros of~$f$ then precisely correspond to the Radon pairs for~$X$ by Lemma~\ref{lem:radon}. Here we use that by the choice of~$\varepsilon$ if $f(v) = -f(w)$ then $v$ and~$w$ are not in a common face of~$\Sigma$. By Theorem~\ref{thm:ramos}, $f$ has a zero. Now let $\varepsilon$ go to zero and use compactness of~$S^d \times S^{d-1}$. 
\end{proof}

Since $S^d$ parametrizes ordered partitions into halfspaces in~$\R^d$ (including the trivial $(\emptyset, \R^d)$ and~$(\R^d, \emptyset)$), the product $S^d \times S^{d-1}$ parametrizes pairs of ordered partitions into halfspaces in~$\R^d$, where the second partition is induced by a hyperplane through the origin. We can thus apply Theorem~\ref{thm:fan-prod} to prove results about mass partitions by two hyperplanes. We will now elaborate on one such example, prove a colorful generalization of Theorem~\ref{thm:fan-prod}, and thus eventually derive a colorful generalization of the mass partition result below.

A pair of partitions into halfspaces $(H_1^+, H_1^-)$ and~$(H_2^+, H_2^-)$ induces four (possibly empty) orthants $H_1^+ \cap H_2^+$, $H_1^+ \cap H_2^-$, $H_1^- \cap H_2^+$, $H_1^- \cap H_2^-$. For a mass~$\mu$ on~$\R^d$ a pair of hyperplanes $H_1$ and $H_2$ induces a \emph{chessboard partition} if 
\[
    \mu(H_1^+ \cap H_2^+)+ \mu(H_1^- \cap H_2^-) = \mu(H_1^+ \cap H_2^-) + \mu(H_1^- \cap H_2^+).
\]
The function $\rho_\mu\colon S^d \times S^d \to \R$ defined by
\[
    \rho_\mu(H_1, H_2) = (\mu(H_1^+ \cap H_2^+)+ \mu(H_1^- \cap H_2^-)) - (\mu(H_1^+ \cap H_2^-) + \mu(H_1^- \cap H_2^+))
\]
measures the extent to which $(H_1, H_2)$ fails to be a chessboard partition of~$\mu$. The map~$\rho_\mu$ is $(\Z/2)^2$-equivariant. 

Barba, Pilz, and Schnider~\cite{barba2019sharing} show that for any four masses on~$\R^2$ there are two hyperplanes (i.e., lines) that simultaneously form a chessboard partition for all four masses. They present a conjecture of Langerman that this should hold more generally for $2d$ masses on~$\R^d$. In fact, Langerman's conjecture is even more general and asserts that the analogous chessboard partition should hold for $nd$ masses on~$\R^d$ and $n$ hyperplanes. This was proved by Hubard and Karasev~\cite{HUBARD_KARASEV_2020}, provided that $d$ is a power of two. Here we prove:

\begin{thm}
\label{thm:fan-chess}
    Let $d=2^t$ be a power of two. Let $m$ be a positive integer, and let $X = \{x_1, \dots, x_m\} \subset \R^{2d-2}$ be a set of $m$ points. Let $\mu_1, \dots, \mu_m$ be masses on~$\R^d$, and suppose that there is a $c>0$ such that for every pair of hyperplanes~$H_1$ and~$H_2$ there is a $\ell\in [m]$ such that $|\rho_{\mu_\ell}(H_1,H_2)|\ge c$. Then there is a pair of hyperplanes~$H_1$ and $H_2$ and disjoint subsets $S, T \subset [m]$ such that $\conv\{x_i\, |\, i\in S\} \cap \conv\{x_i\, |\, i \in T\}\ne \emptyset$ and
    \[
    \rho_{\mu_i}(H_1, H_2) \ge c \ \text{for all} \ i \in S \ \text{and} \ \rho_{\mu_i}(H_1, H_2) \le -c \ \text{for all} \ i \in T.
    \]
\end{thm}

\begin{proof}
    Let $A_i = \{(H_1, H_2) \in S^d \times S^{d-1}\, |\, \rho_{\mu_i}(H_1,H_2) \ge c\}$ and use Theorem~\ref{thm:fan-prod}.
\end{proof}

For $m=2d-1$ no Radon pair exists and so Theorem~\ref{thm:fan-chess} implies that for masses $\mu_1, \dots, \mu_{2d-1}$ on~$\R^d$, $d$ a power of two, there are two hyperplanes~$H_1$ and~$H_2$ that form chessboard partitions of all~$\mu_i$ and so that~$H_2$ passes through the origin. In particular, we do not recover the result of Hubard and Karasev. This is unsurprising since to add another mass~$\mu_{2d}$ and trade it for the restriction that $H_2$ no longer needs to pass through the origin, Hubard and Karasev have to give a subtle argument that relies on the equivariant map having coordinates~$\rho_{\mu_i}$ that come from actual masses~$\mu_i$.

With the same changes as in the proof of Theorem~\ref{thm:fan-gen-col} to modify the proof of Theorem~\ref{thm:fan-prod}, we obtain:

\begin{thm}
\label{thm:fan-prod-col}
    Let $d=2^t$ be a power of two. Let $m$ be a positive integer, and let $X = \{x_1, \dots, x_m\} \subset \R^{2d-2}$ be a set of $m$ points. Let $A_1^{(j)}, \dots, A_m^{(j)} \subset S^d \times S^{d-1}$ be closed sets for $j \in [2d]$ such that $A_i^{(j)} \cap g_1\cdot A_i^{(k)} =\emptyset$, $A_i^{(j)} \cap g_2 \cdot A_i^{(k)} = \emptyset$ for all $i \in [m]$ and $j \ne k \in [2d]$. Further assume that $\bigcup_i G\cdot A_i^{(j)} = S^d \times S^{d-1}$ for all $j\in [2d]$. Then there are disjoint subsets $S, T \subset [m]$, functions $\alpha\colon S \to \{g_1, g_2\}$ and $\beta \colon T\to \{1, g_1g_2\}$, and an injective function $\pi\colon S\cup T \to [2d]$ such that \[\conv\{x_i\, |\, i \in S\} \cap \conv\{x_i\, |\, i \in T\} \ne \emptyset \ \text{and} \ \bigcap_{i \in S} (\alpha(i)\cdot A_i^{(\pi(i))}) \cap \bigcap_{i \in T} (\beta(i)\cdot A_i^{(\pi(i))}) \ne \emptyset.\]
\end{thm}

By repeating the proof of Theorem~\ref{thm:fan-chess} but using Theorem~\ref{thm:fan-prod-col} instead of Theorem~\ref{thm:fan-prod} we derive the following colorful generalization:

\begin{thm}
\label{thm:fan-chess-col}
    Let $d=2^t$ be a power of two. Let $m$ be a positive integer, and let $X = \{x_1, \dots, x_m\} \subset \R^{2d-2}$ be a set of $m$ points. Let $\mu_1^{(j)}, \dots, \mu_m^{(j)}$ be masses on~$\R^d$ for $j\in [2d]$, and suppose that there is a $c>0$ such that for every pair of hyperplanes~$H_1$ and~$H_2$ and for every $j\in [2d]$ there is an $\ell\in [m]$ such that $\rho_{\mu_\ell^{(j)}}(H_1, H_2) \ge c$ or there is a $k \in [m]$ such that $\rho_{\mu_k^{(j)}}(H_1, H_2) \le -c$ but not both. Then there is a pair of hyperplanes~$H_1$ and $H_2$, disjoint subsets $S, T \subset [m]$, and an injective function $\pi\colon S \cup T \to [2d]$ such that $\conv\{x_i\, |\, i\in S\} \cap \conv\{x_i\, |\, i \in T\}\ne \emptyset$ and
    \[
    \rho_{\mu_i^{(\pi(i))}}(H_1, H_2) \ge c \ \text{for all} \ i \in S \ \text{and} \ \rho_{\mu_i^{(\pi(i))}}(H_1, H_2) \le -c \ \text{for all} \ i \in T.
    \]
\end{thm}

\section{Final remarks}
\label{sec:final}

We phrase our main results in terms of Radon-type intersection results for two reasons:
\begin{compactenum}[(1)]
    \item In this phrasing, Fan's theorem is an immediate consequence of Gale's evenness criterion and
    \item this allows for a simple transfer of results from Radon-type results, and more generally non-embeddability results for simplicial complexes, which are numerous, to other problem areas that are approached via the Borsuk--Ulam theorem. 
\end{compactenum}

One can ask more generally what are all sign patterns as in Theorem~\ref{thm:fan-gen} that can be prescribed such that at least one of them yields a non-empty intersection. More precisely, call a $\Z/2$-equivariant subcomplex $K \subset (\Delta_{m-1})^{*2}_\Delta$ \emph{Fan} (for~$S^d$) if for all closed sets $A_1, \dots, A_m \subset S^d$ with $A_i \cap (-A_i) = \emptyset$ for all $i \in [m]$ and $S^d = \bigcup_i A_i \cup \bigcup_i (-A_i)$, there are disjoint faces $\sigma$ and $\tau$ of $\Delta_{m-1}$ such that $\sigma \cup(-\tau)$ is a face of~$K$ with $\bigcap_{i\in \sigma} A_i \cap \bigcap_{i \in \tau} (-A_i) \ne \emptyset$. In the following, we will fix~$d$ and call a complex $K$ Fan if it is Fan for~$S^d$. Theorem~\ref{thm:fan-gen} shows that for any set of $m$ points in~$\R^{d-1}$ the subcomplex of $(\Delta_{m-1})^{*2}_\Delta$ consisting of the downward closure of Radon pairs $(A,B)$ is a Fan complex. This complex of Radon pairs is a sphere of dimension~$m-d-1$, and -- as a consequence of Lemma~\ref{lem:radon} -- the complex of pairs that fail to be Radon is a sphere of dimension~$d-1$. The proof of Theorem~\ref{thm:fan-gen-cont} more generally works to show:

\begin{thm}
\label{thm:zeroset}
    Let $K \subset (\Delta_{m-1})^{*2}_\Delta$ be a $\Z/2$-equivariant subcomplex. If there is a $\Z/2$-map $f\colon (\Delta_{m-1})^{*2}_\Delta \to \R^d$ with $K = f^{-1}(0)$ then $K$ is Fan.
\end{thm}

\begin{proof}
    Repeat the proof of Theorem~\ref{thm:fan-gen-cont} producing a map $\widetilde \alpha \colon S^d \to (\Delta_{m-1})^{*2}_\Delta$ from the covering by $\pm A_1, \dots, \pm A_m$. Then use the map $f$ in place of~$\widetilde h$ and conclude in the same way.
\end{proof}

The condition that $K$ is the zero-set of a $\Z/2$-equivariant map to~$\R^d$ is equivalent to the complement of $K$ in $(\Delta_{m-1})^{*2}_\Delta$ mapping $\Z/2$-equivariantly to~$S^{d-1}$. Let $K^c$ be the induced subcomplex of the barycentric subdivision of~$(\Delta_{m-1})^{*2}_\Delta$ whose vertices subdivide faces of $(\Delta_{m-1})^{*2}_\Delta$ that are not faces of~$K$. Let $h\colon K^c \to S^{d-1}$ be a $\Z/2$-equivariant map. Define $f\colon (\Delta_{m-1})^{*2}_\Delta\to \R^d$ as constantly zero on~$K$ as $h$ on~$K^c$, and otherwise interpolate linearly. Then $K = f^{-1}(0)$. Conversely, if $f\colon (\Delta_{m-1})^{*2}_\Delta\to \R^d$ is a $\Z/2$-equivariant map with $K = f^{-1}(0)$, then by restriction $f$ induces a $\Z/2$-equivariant map $K^c \to \R^d \setminus \{0\}$, whose image may be normalized to be in~$S^{d-1}$.

The Borsuk--Ulam theorem provides a simple witness for the non-existence of a $\Z/2$-equivariant map $K^c \to S^{d-1}$: A $\Z/2$-equivariant map $h\colon S^d \to K^c$. Recall that $K^c$ is a subcomplex of the barycentric subdivision of~$(\Delta_{m-1})^{*2}_\Delta$. Identify the vertex set of~$(\Delta_{m-1})^{*2}_\Delta$ with $\{\pm 1, \dots, \pm m\}$. For $v \in \{\pm 1, \dots, \pm m\}$ let $S_v$ be the subcomplex of $K^c$ consisting of all faces $\sigma$ such that $\sigma \cup \{v\}$ is a face of~$K^c$. Let $A_v = h^{-1}(S_v)$. By $\Z/2$-symmetry $-A_v = A_{-v}$. The sets $A_1, \dots, A_m \subset S^d$ witness that $K$ is not Fan: Indeed, let $\sigma$ and $\tau$ be disjoint subsets of $[m]$ with $\bigcap_{v \in \sigma} A_v \cap \bigcap_{w\in \tau} A_{-w} \ne \emptyset$. Then the intersection $\bigcap_{v \in \sigma} S_v \cap \bigcap_{w\in \tau} S_{-w}$ contains a vertex of~$K^c$. This vertex subdivides a face that contains $\sigma$ and~$-\tau$, and in particular $\sigma\cup(-\tau)$ is not a face of~$K$.
We thus have the following:

\begin{thm}
\label{thm:notFan}
    Let $K \subset (\Delta_{m-1})^{*2}_\Delta$ be a $\Z/2$-equivariant subcomplex. If there is a $\Z/2$-map $K^c \to S^{d-1}$ then $K$ is Fan. If there is a $\Z/2$-map $S^d \to K^c$ then $K$ is not Fan.
\end{thm}

A space $X$ with a free $\Z/2$-action is called \emph{tidy} if for the largest~$d$ such that there is a $\Z/2$-equivariant map $S^d \to X$, there is a $\Z/2$-equivariant map $X\to S^d$. By Theorem~\ref{thm:notFan}, Theorem~\ref{thm:zeroset} characterizes Fan complexes among all complexes~$K \subset (\Delta_{m-1})^{*2}_\Delta$ where $K^c$ is a tidy space. Non-tidy spaces exist~\cite{csorba2005non, matsushita2014some}.

Lastly, we show that Theorem~\ref{thm:fan-gen} is optimal in the sense that it no longer holds if a single Radon pair is disallowed.

\begin{thm}
    Let $X = \{x_1, \dots, x_m\} \subset \R^{d-1}$ be a generic point set. Let $(A,B)$ be a minimal Radon pair for~$X$. Let $K \subset (\Delta_{m-1})^{*2}_\Delta$ consist of all Radon pairs for~$X$ with the exception of $(A,B)$ and~$(B,A)$. Then $K$ is not Fan.
\end{thm}

\begin{proof}
    We identify the vertex set of~$(\Delta_{m-1})^{*2}_\Delta$ with $\{\pm x_1, \dots, \pm x_m\}$. Let $C$ and $D$ be two disjoint subsets of $A \cup B$. Then since $(A,B)$ is a minimal Radon pair, no other partition of $A\cup B$ is a Radon pair, and so $(C,D)$ is either not a Radon pair or $(C,D) \in \{(A,B), (B,A)\}$. In particular, $(C,D)$ is a vertex of~$K^c$, and $K^c$ contains the induced subcomplex of all these vertices in the barycentric subdivision of~$(\Delta_{m-1})^{*2}_\Delta$. Since $A\cup B$ involves $d+1$ vertices, this subcomplex is isomorphic to the barycentric subdivision of~$(\Delta_d)^{*2}_\Delta$, which is a $d$-sphere. Theorem~\ref{thm:notFan} finishes the proof.
\end{proof}

\section*{Acknowledgments}
The authors would like to thank Pablo Sober\'{o}n for helpful comments.


\end{document}